\documentclass[11pt]{article}
\usepackage{amssymb, amsthm, amsmath, amscd}
\usepackage[OT2, OT1]{fontenc}
\newcommand\cyr{%
\renewcommand\rmdefault{wncyr}%
\renewcommand\sfdefault{wncyss}%
\renewcommand\encodingdefault{OT2}%
\normalfont
\selectfont}
\DeclareTextFontCommand{\textcyr}{\cyr}

\usepackage[all]{xy}
\usepackage{mathrsfs}

\newtheorem{thm}{Theorem}[section]
\newtheorem{lem}[thm]{Lemma}
\newtheorem{prop}[thm]{Proposition}
\newtheorem{cor}[thm]{Corollary}
\newtheorem{NN}[thm]{}
\theoremstyle{definition}\newtheorem{df}[thm]{Definition}
\theoremstyle{definition}\newtheorem{rem}[thm]{Remark}
\theoremstyle{definition}\newtheorem{exm}[thm]{Example}

\renewcommand{\phi}{\varphi}

\newcommand{\N}{\mathbb{N}}
\newcommand{\Z}{\mathbb{Z}}
\newcommand{\Q}{\mathbb{Q}}
\newcommand{\R}{\mathbb{R}}

\newcommand{\T}{\mathbb{T}}

\newcommand{\hm}{homomorphism}
\newcommand{\dt}{\delta}
\newcommand{\ep}{\epsilon}
\newcommand{\andeqn}{\,\,\,{\rm and}\,\,\,}
\newcommand{\rforal}{\,\,\,{\rm for\,\,\,all}\,\,\,}
\newcommand{\CA}{$C^*$-algebra}
\newcommand{\SCA}{$C^*$-subalgebra}

\newcommand{\af}{{\alpha}}
\newcommand{\bt}{{\beta}}

\newcommand{\beq}{\begin{eqnarray}}
\newcommand{\eneq}{\end{eqnarray}}
\newcommand{\tforal}{\,\,\,\text{for\,\,\,all}\,\,\,}

\newcommand{\fp}{{\mathfrak{p}}}
\newcommand{\fq}{{\mathfrak{q}}}
\title{The Range of a Class of Classifiable Separable Simple Amenable
\CA s
}
\author{Huaxin Lin and Zhuang Niu
 }
\date{}

\begin{document}

\maketitle

\begin{abstract}

We study the range of a classifiable class ${\cal A}$ of unital
separable simple amenable \CA s which satisfy the Universal
Coefficient Theorem.  The class ${\cal A}$ contains all unital
simple AH-algebras. We show that all unital simple inductive limits
of dimension drop circle \CA s are also in the class. This unifies
some of the previous known classification results for unital simple
amenable \CA s. We also show that there are many other \CA s in the
class. We prove  that, for any partially ordered, simple weakly
unperforated rationally Riesz group $G_0$ with order unit $u,$ any
countable abelian group $G_1,$ any metrizable Choquet simplex $S,$
and any surjective affine continuous map  $r: S\to S_u(G_0)$ (where
$S_u(G_0)$ is the state space of $G_0$) which preserves extremal
points, there exists one and only one (up to isomorphism) unital
separable simple amenable \CA\, $A$ in the classifiable class ${\cal
A}$ such that
$$
((K_0(A), K_0(A)_+, [1_A]), K_1(A), T(A), \lambda_A)=((G_0, (G_0)_+, u),
G_1,S, r).
$$
%A characterization (rational Reisz group) of $K_0(A)$
%for $A\in {\cal A}$
%is given.  We show that any possible $(G_0, G_1)$ can be in fact exhausted  by $K$-theory of \CA s in ${\cal A}.$

\end{abstract}

\section{Introduction}

Recent years saw some rapid development in the theory of
classification of amenable \CA s, or otherwise know as the Elliott
program of classification of amenable \CA s. One of the high lights
is the Kirchberg-Phillips's  classification of separable purely
infinite simple amenable \CA s which satisfy the Universal
Coefficient Theorem (see \cite{Ph1} and  \cite{KP0}). There are also
exciting results in the simple \CA s of stable rank one. For
example, the classification of unital simple AH-algebra with no
dimension growth by Elliott, Gong and Li (\cite{EGL-AH}). Limitation
of the classification have been also discovered (see
\cite{Ror-infproj} and \cite{Toms-Ann}, for example).  In
particular, it is now known that the general class of unital simple
AH-algebras can not be classified by the traditional Elliott
invariant. One crucial condition must be assumed for any general
classification (using the Elliott invariant) of separable simple
amenable \CA s is the ${\cal Z}$-stability. On the other hand,
classification theorems were established for unital separable simple
amenable \CA s which are not assumed to be AH-algebras, or other
inductive limit structures (see \cite{Lnduke}, \cite{LinTAI} and
\cite{Niu-thesis}). Winter's recent result provided a new approach to some
more general classification theorem (\cite{Winter-Z} and
\cite{Lin-App}). Let ${\cal A}$ be the class of unital separable
simple amenable \CA s $A$ which satisfy the UCT for which $A\otimes
M_\fp$ has tracial rank no more than one for some supernatural
number $\fp$ of infinite type. A more recent work in \cite{Lnclasn}
shows that \CA s in ${\cal A}$ can be classified by the Elliott
invariant up to ${\cal Z}$-stable isomorphism.  All unital simple
AH-algebras are in ${\cal A}.$ One consequence of this is now we
know that classifiable class of unital simple AH-algebras is exact
the class of those of ${\cal Z}$-stable ones. But class ${\cal A}$
contains more unital simple \CA s. Any unital separable simple
ASH-algebra $A$ whose state space $S(K_0(A))$ of its $K_0(A)$ is the
same as that tracial state space are in ${\cal A}.$ It also contains
the Jiang-Su algebra ${\cal Z}$ and many other projectionless simple
\CA s.  We show that the class contains all unital simple so-called
dimension drop circle algebras as well as many other \CA s whose
$K_0$-groups are not Riesz. It is  the purpose of this paper to
discuss the range of invariant of \CA s in ${\cal A}.$

{\bf Acknowledgement:} The work of the first named author is
partially supported by a grant of NSF, Shanghai Priority Academic
Disciplines and Chang-Jiang Professorship from East China Normal
University during the summer 2008. The work of the second named author is supported by an NSERC Postdoctoral Fellowship.

\section{Preliminaries}

\begin{df}[Dimension drop interval algebras \cite{JS-Z}] A dimension drop interval algebra is a \CA\, of the form:
$$
\mathbf{I}(m_0, m,m_1)=\{f\in \mathrm{C}([0,1], M_m): f(0)\in M_{m_0}\otimes 1_{m/m_0}\andeqn
f(1)\in 1_{m/m_1}\otimes M_{m_1}\},
$$
where $m_0,m_1$ and $m$ are positive integers with $m$ is divisible
by $m_0$ and $m_1.$ If $m_0$ and $m_1$ are relatively prime, and
$m=m_0m_1,$ then $\mathbf{I}(m_0, m,m_1)$ is called a prime
dimension drop algebra.
\end{df}

\begin{df}[The Jiang-Su algebra \cite{JS-Z}]
Denote by ${\cal Z}$ the Jiang-Su algebra of unital infinite dimensional simple
\CA\, which is an inductive limit of prime dimension drop algebras with a unique tracial state,
$(K_0({\cal Z}), K_0({\cal Z})_+, [1_{\cal Z}])=(\Z, \N, 1)$ and $K_1({\cal Z})=0.$

\begin{df}[Dimension drop circle algebras \cite{Myg-CJM}] Let $n$ be a natural number. Let $x_1, ..., x_N$ be points in the circle $\mathbb T$, and let $d_1, ..., d_N$ be natural numbers dividing $n$. Then a dimension drop circle algebra is a \CA\ of the form
$$
D(n, d_1 ,..., d_N)=\{f\in \mathrm{C}(\mathbb T, \mathrm{M}_N):\ f(x_i)\in M_{d_i}\otimes 1_{n/d_i}, i=1,2,...,N\}.
$$
\end{df}

\end{df}
\begin{df}[Simple A$\T$D-algebras]
By A$\T$D-algebras we mean \CA s which are inductive limits of dimension drop circle algebras.
\end{df}

\begin{df}Denote by ${\cal I}$ the class of those \CA s with the
form $\bigoplus_{i=1}^n M_{r_i}(C(X_i)),$ where each $X_i$ is a finite
CW complex with (covering) dimension no more than one.

A unital simple \CA\, $A$ is said to have tracial rank one if
for any finite subset $\mathcal F\subset A$, $\ep>0$, any nonzero
positive element $a\in A$, there is a \SCA\,  $C\in\mathcal I$ such
that if denote by $p$ the unit of $C$, then for any $x\in\mathcal
F$, one has
\begin{enumerate}
\item $||xp-px||\leq\ep,$
\item there is $b\in C$ such that $||b-pxp||\leq \ep$  and
\item $1-p$ is Murray-von Neumann equivalent to a projection in $\overline{aAa}$.
\end{enumerate}
\end{df}

Denote by ${\cal I}'$ the class of all unital \CA s with the form
$\bigoplus_{i=1}^n M_{r_i}(C(X_i)),$ where each $X_i$ is a compact
metric space with dimension no more than one.

 Note that, in the above
definition, one may replace ${\cal I}$ by ${\cal I}'.$

\begin{df}[A classifiable class of unital separable simple amenable \CA s]

Denote by ${\cal N}$ the class of all unital separable amenable
\CA s which satisfy the Universal Coefficient Theorem.

For a supernatural number $\mathfrak p$, denote by $M_\fp$ the UHF
algebra associated with $\fp$ (see \cite{Dix}).

Let ${\cal A}$ denote the class of all unital separable simple amenable
\CA s $A$ in ${\cal N}$ for which $TR(A\otimes M_{\mathfrak{p}})\le
1$ for all  supernatural numbers ${\mathfrak{p}}$ of infinite type.
\end{df}

\begin{rem}
By Theorem \ref{T0} below, in order to verify whether a C*-algebra $A$ is in the class $\mathcal A$, it is enough to verify $TR(A\otimes M_{\mathfrak{p}})\leq1$ for one supernatural number $\mathfrak p$ of infinite type.
\end{rem}

\begin{df}\label{Inf}

{\rm Let $G$ be a partially ordered group with an order unit $u\in
G.$ Denote by $S_u(G)$ the state space of $G,$ i.e., $S_u(G)$ is the
set of all positive \hm s $h: G\to \R$ such that $h(u)=1.$ The set
$S_u(G)$ equipped with the weak-*-topology forms a compact convex
set. Denote by $\mathrm{Aff}(S_u(G))$ the space of all continuous
real affine functions on $S_u(G).$ We use $\rho$ for the \hm\,
$\rho: G\to \mathrm{Aff}(S_u(G))$ defined by
$$
\rho(g)(s)=s(g)\tforal s\in S_u(G)\andeqn \tforal g\in G.
$$
Put $\mathrm{Inf}(G)={\rm ker}\rho.$

}
\end{df}
\begin{df}
Let $A$ be a unital stably finite separable simple amenable \CA.
Denote by $T(A)$ the tracial state space of $A.$ We also use $\tau$
for $\tau\otimes {\rm Tr}$ on $A\otimes M_k$ for any integer $k\ge
1,$ where $Tr$ is the standard trace on $M_k.$

By $\mathrm{Ell}(A)$ we mean the following:
$$
(K_0(A), K_0(A)_+, [1_A], K_1(A), T(A), r_A),
$$
where $r_A: T(A)\to S_{[1_A]}(K_0(A))$ is a surjective continuous
affine map such that $r_A(\tau)([p])=\tau(p)$ for all projections
$p\in A\otimes M_k,$ $k=1,2,....$

Suppose that $B$ is another stably finite unital separable simple
\CA. A map $\Lambda: \mathrm{Ell}(A)\to \mathrm{Ell}(B)$ is said to
be a \hm\, if $\Lambda$ gives an order \hm\, $\lambda_0: K_0(A)\to
K_0(B)$ such that $\lambda_0([1_A])=[1_B],$ a \hm\, $\lambda_1:
K_1(A)\to K_1(B),$ a continuous affine map $\lambda_{\rho}': T(B)\to
T(A)$ such that
$$
\lambda_{\rho}'(\tau)(p)=r_B(\tau)(\lambda_0([p]))
$$
for all projection in $A\otimes M_k,$ $k=1,2,...,$ and for all
$\tau\in T(B).$

We say that such $\Lambda$ is an isomorphism, if $\lambda_0$ and
$\lambda_1$ are isomorphisms and $\lambda_\rho'$ is a affine
homeomorphism. In this case, there is an affine homeomorphism
$\lambda_\rho: T(A)\to T(B)$ such that
$\lambda_\rho^{-1}=\lambda_\rho'.$

\end{df}

\begin{thm}[Corollary 11.9 of \cite{Lnclasn}]
Let $A,\, B\in {\cal A}.$ Then
$$
A\otimes {\cal Z}\cong B\otimes {\cal Z}
$$
if
$$
\mathrm{Ell}(A\otimes {\cal Z})=\mathrm{Ell}(B\otimes {\cal Z}).
$$
\end{thm}

 In the next section (Theorem \ref{T1}), we will show the
 following:

\begin{thm}\label{T0}  Let $A$ be a unital separable amenable simple \CA. Then
$A\in {\cal A}$ if only if    $TR(A\otimes M_\fp)\le 1$ for one
supernatural number $\fp$ of infinite type
\end{thm}

\begin{df}
{\rm Recall that a \CA\, $A$ is said to be ${\cal Z}$-stable if
$A\otimes {\cal Z}\cong A.$ Denote by ${\cal A}_{\cal Z}$ the class
of  ${\cal Z}$-stable  \CA s in ${\cal A}.$ Denote by ${\cal
A}_{0z}$ the subclass of those \CA s $A\in {\cal A}_z$ for which
$TR(A\otimes M_\fp)=0$ for some supernatural number $\fp$ of
infinite type.

}
\end{df}

\begin{cor}\label{CT0}
Let $A$ and $B$ be two unital separable amenable simple \CA s in
${\cal A}_{\cal Z}.$  Then $A\cong B$ if and only if
$$
\mathrm{Ell}(A)\cong \mathrm{Ell}(B).
$$
\end{cor}

\begin{df}\label{dflesssim}
{\rm Let $A$ be a unital simple \CA\, and let $a, b\in A_+.$ We
define
$$
a\lesssim b,
$$
if there exists $x\in A$ such that $x^*x=a$ and $xx^*\in
\overline{bAb}.$  We write $[a]=[b]$ if there exists $x\in A$ such
that $x^*x=a$ and $xx^*=b.$ We write $[a]\le [b]$ if $a\lesssim
b.$

If $e\in A$ is a projection and $[e]\le [a],$ then there is a
projection $p\in\overline{aAa}$ such that $e$ is equivalent to
$p.$

If there $n$ mutually orthogonal elements $b_1,b_2,...,b_n\in
\overline{bAb}$ such that $a\lesssim b_i,$ $i=1,2,...,n,$ then we
write
$$
n[a]\le [b].
$$
Let $m$ be  another positive integer. We write
$$
n[a]\le m[b],
$$
if, in $M_m(A),$
$$
n[a]\le [\begin{pmatrix} b & 0 & 0 \cdots  &0\\
                       0 & b & 0\cdots & 0\\
                        \vdots&\vdots& \ddots  &\vdots\\

                       0 & 0 &0 \cdots  & b\cr
                       \end{pmatrix}],
                       $$
where $b$ repeats $m$ times.

}
\end{df}

\section{The classifiable \CA s ${\cal A}$}

The  purpose of this section is to provide a proof of Theorem
\ref{T0}.

\begin{lem}\label{A1}
Let $A$ be a unital simple \CA\, with an increasing sequence of
unital simple \CA s $\{A_n\}$ such that $1_A=1_{A_n}$ and
$\cup_{n=1}^{\infty} A_n$ is dense in $A.$

Suppose that $a\in A_+\setminus\{0\}.$ Then, there exists $b\in
(A_n)_+\setminus\{0\}$ for some large $n$ so that $b\lesssim a.$
\end{lem}

\begin{proof}
Without loss of generality, we may assume that $\|a\|=1.$  Let
$1>\dt>0.$ There exists $\ep>0$ such that for any $c\in
A_+\setminus \{0\}$ with
$$
\|a-c\|<\ep
$$
one has (see Proposition 2.2 of \cite{RorUHF2}) that
$$
f_\dt(c)\lesssim a,
$$
where $f_\dt\in C_0((0, \infty))$ for which $f_\dt(t)=1$ for all
$t\ge \dt$ and $f_\dt(t)=0$ for all $t\in (0, \dt/2).$ We may
assume, for a sufficiently small $\ep,$ $f_\dt(c)\not=0.$ Since
$\cup_{n=1}^{\infty} A_n$ is dense in $A,$ it is possible to such
that $c\in A_n$ for some large $n.$ Put $b=f_\dt(c).$ Then $b\in
A_n$ and $b\lesssim a.$

\end{proof}

\begin{lem}\label{AL0} Let $A$ be a unital  simple \CA\, and $e,
a\in A_+\setminus\{0\}.$ Then, for any integer $n>0,$ there exists
$m(n)>0$ such that
$$
n[e]\le{ m(n)}[a].
$$
\end{lem}

\begin{proof}
It suffices to show that $[1_A]\le m[a]$ for some integer $m.$
Since $A$ is simple, there are $y_1,y_2,...,y_m\in A$ for some
integer $m$ such that
$$
\sum_{i=1}^m y_i^*ay_i=1
$$
(see, for example, Lemma 3.3.6 of \cite{Lin-book}). Let
$b_i'=a^{1/2}y_iy_i^*a^{1/2},$ $i=1,2,...,m$ Define
$$
b_i={\rm diag}(\overbrace{0, 0,..., 0}^{i-1},
b_i,0,...,0),\,\,\,i=1,2,...,m
$$
in $M_m(A).$ Define
$$
y=\begin{pmatrix} y_1^*a^{1/2} & y_2^*a^{1/2} & \cdots & y_m^*a^{1/2}\\
                   0 & 0& \cdots & 0\\
                   \vdots &\vdots &\ddots &\vdots \\
                      0 & 0& \cdots & 0\end{pmatrix}.
                      $$
Then
$$
yy^*=1_A\andeqn y^*y=\sum_{i=1}^m b_i.
$$
Thus
$$
[1_A]\le [\sum_{i=1}^m b_i].
$$
Clearly
$$
[\sum_{i=1}^m b_i]\le m[a].
$$

\end{proof}

\begin{lem}\label{AL00} Let $A$ be a unital infinite dimensional simple
\CA\, and let $a\in A_+\setminus\{0\}.$ Suppose that $n\ge 1$ is
an integer. Then there are non-zero mutually orthogonal elements
$a_1, a_2,...,a_n\in \overline{aAa}_+$ such that
$$
a_1\lesssim a_2\lesssim \cdots \lesssim a_n.
$$
\end{lem}

\begin{proof}
In $\overline{aAa},$ there exists a positive element $0\le b\le 1$
such that its spectrum has infinitely many points. From this one
obtains $n$ non-zero mutually orthogonal positive elements
$b_1,b_2,...,b_n.$ Since $A$ is simple, there is $x_{n-1}\in A$ such that $b_{n-1}x_{n-1}b_n\neq 0$. Set $y_{n-1}=b_{n-1}x_{n-1}b_n$. One then has that $y_{n-1}^*y_{n-1}\in \overline{b_nAb_n}$ and
$y_{n-1}y_{n-1}^*\in \overline{b_{n-1}Ab_{n-1}}.$ Then consider
$b_1,b_2,...,b_{n-2}, y_{n-1}y_{n-1}^*.$ The lemma follows by applying the argument above
$n-1$ more times.
\end{proof}

\begin{lem}\label{AL1}
 Let $A$ be a unital separable simple \CA.
 Suppose that $\{A_n\}$ be an increasing sequence of
unital simple \CA s such that $1_A=1_{A_n}$ and
$\cup_{n=1}^{\infty} A_n$ is dense in $A.$

Then $TR(A)\le 1$ if and only if the following holds: For any
$\ep>0,$ any integer $n\ge 1,$ any $a\in (A_n)_+\setminus \{0\}$ and
a finite subset ${\cal F}\subset A_n,$ there exists an integer $N\ge
1$ satisfying the following: There is, for each $m\ge N,$  a \SCA\,
$C\subset A_m$ with $C\in {\cal I}'$ and with $1_C=p$ such that

{\rm (1)} $\|px-xp\|<\ep\tforal x\in {\cal F};$

{\rm (2)} ${\rm dist}(pxp, C)<\ep\tforal x\in {\cal F}$ and

{\rm (3)} $1-p$ is equivalent to a projection $q\in
\overline{aA_ma}.$

\end{lem}

\begin{proof}

First we note that ``if " part follows the definition easily.

To prove the ``only if " part, we assume that $TR(A)\le 1$ and there
is an increasing sequence of \SCA s $A_n\subset A$ which are unital
and the closure of the union $\cup_{n=1}^{\infty}A_n$ is dense in
$A.$ In particular, we may assume that $1_{A_n}=1_A,$ $n=1,2,....$

Let $\ep>0,$ $n\ge 1,$ $a\in (A_n)_+\setminus \{0\}$ and a finite
subset ${\cal F}\subset A_n.$ Since $TR(A)\le 1,$ there is a \SCA\,
$C_1\in {\cal I}$ with $1_{C_1}=q$ such that

(1) $\|xq-qx\|<\ep/2$ for all $x\in {\cal F};$

(2) ${\rm dist}(qxq,C_1)<\ep/2$ for all $x\in {\cal F}$ and

(3) $[1-q]\le [a].$

From (2) above, there is a finite subset ${\cal F}_1\subset C_1$
such that

(2') ${\rm dist}(qxq, {\cal F}_1)<\ep/2$ for all $x\in {\cal F}.$

We may assume that $q\in {\cal F}_1.$

Put $d=\sup\{\|x\|; x\in {\cal F}\}.$

   There is $\dt>0$ and a finite subset ${\cal G}\subset C_1$ satisfying the following: If $B\subset A$ is a
   unital \SCA\, and
   $$
   {\rm dist}(y, B)<\dt
   $$ for all $y\in {\cal G},$ there is a \SCA\, $C'\in {\cal I}'$
   such that $C'\subset B$ and
   $$
   {\rm dist}(z, C')<\ep/2(d+1)
   $$
   for all $z\in {\cal F}_1\cup {\cal G}.$

   By the assumption, there is an integer $N\ge n$ such that
   $$
   {\rm dist}(y, A_N)<\dt/2
   $$
   for all $y\in {\cal G}.$ It follows that there is $C\in {\cal
   I}'$ with $1_C=p$ such that $C\in A_N$ and
   $$
   {\rm dist}(z, C)<\ep/2(d+1)
   $$
   for all $z\in {\cal F}_1\cup {\cal G}.$ In particular,
   $$
   \|p-q\|<\ep/2(d+1).
   $$
   One then checks that

(i) $\|qx-xq\|<\ep$ for all $x\in {\cal F};$

(ii) ${\rm dist}(pxp, C)<\ep$ for all $x\in {\cal F}$ and

(iii) $[1-p]=[1-q]\le [a].$

The lemma follows.

\end{proof}

{The following is well-known.}

\begin{lem}\label{Ldiv}
 For any two positive integers $p$ and $q$ with $p>q.$
Then there exist integers  $1\le m<q,$ $0\le r<p$ and $s\ge 1$ such that
$$
p=mq^s+r.
$$
\end{lem}

\begin{proof}
There exists a largest integer $s\ge 1$ such that $p\ge q^s.$
There are integers $0\le r< p$ and $1\le m$ such that
$$
p=mq^s+r.
$$
Then $1\le m<q.$

\end{proof}

Theorem \ref{T0} follows immediately from the following:

\begin{thm}\label{T1} Let $A$ be a unital separable simple \CA. Then
$TR(A\otimes M_\fp)\le 1$ for all  supernatural numbers $\fp$ of infinite type
%unital infinite dimensional
%UHF-algebra $u$
 if and only if there exists one supernatural number $\fq$ of infinite type
 %unital infinite dimensional
% UHF-algebra $U_0$
 such that $TR(A\otimes M_\fq)\le 1.$
\end{thm}

\begin{proof}
Suppose that there is a supernatural number $\fq$ of infinite type
%a unital infinite dimensional UHF-algebra $U_0$
such that $TR(A\otimes M_\fq)\le 1.$

Let $A_n\cong A\otimes M_{r(n)}$  such that $1_{A_n}=1_{A\otimes
M_\fp}$ and $\cup_{n=1}^{\infty} A_n$ is dense in $A\otimes
M_\fp.$ Let $B_n\cong A\otimes M_{k(n)}$ such that
$1_{B_n}=1_{A\otimes M_\fq}$ and $\cup_{n=1}^{\infty} B_n$ is
dense in $A\otimes M_\fp.$

Write $\fq=q_1^{\infty}q_2^{\infty}\cdots q_k^{\infty}\cdots,$
where $q_1,q_2,...,q_k,...$ are prime numbers. We may assume that
$1<q_1<q_2<\cdots q_k<\cdots.$

 Fix $\ep>0,$ $n\ge 1,$ $a_0\in (A_n)_+\setminus\{0\}$ and ${\cal F}\subset
 A_n.$ Since $A_n$ is simple, there are mutually orthogonal
 elements $a_1, a_2,...,a_{3(q_1+1)}\in (A_n)_+\setminus\{0\}$ such that $a_1,
 a_2,..., a_{3(q_1+1)}\in \overline{a_0A_na_0}$ and
 $a_1\lesssim a_2\lesssim \cdots \lesssim a_{3(q_1+1)}.$

 By \ref{A1}, there exists an integer
 %$m(q_1)>0$ such that
 %$$
 %q_1[1_{A_n}]\le m(p_1) [a_{q_1+1}]\,\,\,{\rm in}\,\,\, M_{q_1+m(q_1)}(A_n).
 %$$
 $m(1)\ge 1$ such that
$$
[1_{A_n}]\le m(1)[a_{q_1+1}].
$$

 Write $A_n\cong M_{r(n)}(A).$

 Note that $TR(M_{r(n)}(A)\otimes
 M_\fq)\le 1.$ Therefore, by \ref{AL1},  there exists $N\ge 1$ satisfying the following:
  there is, for each $m\ge N,$  a \SCA\, $C_m\subset
 M_{r(n)}(A)\otimes M_{k(m)}$ with $C_m\in {\cal I}'$ and $1_C=e(m)$ such that

(i) $\|e(m) j_0(x) -j_0(x)e(m)\|<\ep/2$ for all $x\in {\cal F},$ where
$j_m: M_{r(n)}(A)\to M_{r(n)}(A)\otimes M_{k(m)}$ is defined by
$j_m(a)=a\otimes 1_{M_{k(m)}}$ for $a\in M_{r(n)}(A);$

(ii) ${\rm dist}(e(m)j_m(x)e(m), C)<\ep/2$ for all $x\in {\cal F}$ and

(iii) $1_{M_{r(n)}(A)\otimes M_{r(m)}}-e(m)$ is equivalent to a
projection in $\overline{j_m(a_1)(A_n\otimes M_{k(m)})j_m(a_1)}.$

We may assume that  $k(N)=q_1^{s_1}q_2^{s_2}\cdots q_k^{s_k}$ and
$k(N)\ge m(q_1).$

To simplify the notation, without loss of generality, we may
assume that $r(n+1)/r(n)>k(N).$ We write
$$
{r(n+1)\over{r(n)}}=N_1k(N)+r_0,
$$
where  $N_1\ge 1$and  $0\le r_0<k(N)$ are integers.  Without loss
of generality, we may further assume that $N_1>m(1)k(N).$ We write
$$
N_1=n_1q_1^s+r_1,
$$
where $q_1>n_1\ge 1, $ $s\ge 1, $ and $0\le r_1<q_1$ are integers.
Thus
$$
{r(n+1)\over{r(n)}}=n_1q_1^sK(N)+ r_1K(N)+ r_0.
$$
Without loss of generality, to simplify notation, we may assume
that $k(N+1)=q_1^sK(N).$

Put $C=\bigoplus_{j=1}^{n_1} C_{N+1}\oplus \bigoplus_{i=1}^{r_1} C_N.$
Put $e=\bigoplus_{j=1}^{n_1}e(N+1)\oplus \bigoplus_{i=1}^{r_1}e(N).$
Define $j': M_{r(n)}(A)\to M_{r(n)}(A)\otimes M_{r(n+1)/r(n)-r_0}$
by
$$
j'(a)={\rm diag}(\overbrace{j_{N+1}(a), j_{N+1}(a),\cdots,
j_{N+1}(a)}^{n_1}) \oplus {\rm diag}(\overbrace{j_{N}(a),
j_{N}(a),\cdots, j_{N}(a)}^{r_1})
$$
for all $a\in M_{r(n)}(A).$

 By
what we have proved, we have that

$\|ej'(x)-j'(x)e\|<\ep/2$ for all $x\in {\cal F},$

${\rm dist}(ej'(x)e, C)<\ep/2$ for all $x\in {\cal F}$ and

$1_{M_{r(n)}(A)\otimes M_{r(n+1)/r(n)}-r_0}-e$ is equivalent to a
projection in $\overline{b_1(M_{r(n)}(A)\otimes
M_{r(n+1)/r(n)-r_0})b_1},$ where
$$
b_1=\sum_{i=1}^{n_1}j'(a_i)+\sum_{i=p_1+1}^{p_1+r_1}j'(a_i).
$$
The last assertion follows the fact that
$$
n_1[1_{M_{r(n)}(A)\otimes M_{k(N+1)}}-e(N+1)]\le
[\sum_{i=1}^{n_1}a_i]\andeqn
$$
$$
 r_1[1_{M_{r(n)}(A)\otimes
M_{k(N)}}-e(N)]\le [\sum_{i=p_1+1}^{p_1+r_1}a_i].
$$
Define $j: M_{r(n)}(A)\to M_{r(n+1)}(A)=M_{r(n)}(A)\otimes
M_{r(n+1)/r(n)}$ by
$$
j(a)=a\otimes 1_{M_{r(n+1)/r(n)}}\tforal a\in M_{r(n)}.
$$

Thus, in $M_{r(n+1)}(A),$

$\|ej(a)-j(a)e\|<\ep/2$ for all $x\in {\cal F},$

${\rm dist}(ej(a)e, C)<\ep/2$ for all $x\in {\cal F}.$

Note that
$$
r_0[1_{M_{r(n)}}]\le r_0m(1)[a_1]\le [j(a_{2q_1+1})].
$$
Thus
$$
[1_{A_{n+1}}-e]=r_0[1_{M_{r(n)}}]+[1_{M_{r(n)}(A)\otimes
M_{r(n+1)/r(n)}-r_0}-e]\le [\sum_{i=1}^{3(q_1+1)} a_i].
$$
It follows that $1_{A_{n+1}}-e$ is equivalent to a projection in
$\overline{j(a_0)A_{n+1}j(a_0)}.$

This proves that $TR(M_{r(n)}(A)\otimes M_{\fp})\le 1.$ It follows
from Proposition 3.2 of \cite{LinTAI} that $TR(A\otimes M_\fp)\le 1.$

\end{proof}

\begin{cor}\label{C1}
Let $A$ be a unital separable simple \CA. Suppose that there exists
a supernatural number $\fq$ of infinite type for which $TR(A\otimes
M_\fq)=0.$ Then, for all supernatural number $\fp$ of infinite type,
$TR(A\otimes M_\fp)=0.$
\end{cor}

\begin{proof}
The proof uses the exactly the same argument used in the proof of
\ref{T1} (and \ref{AL1}).
\end{proof}

\section{Unital Simple A$\T$D-algebras}

\begin{thm}\label{AD1}
Every unital simple A$\T$D-algebra $A$ for which
$K_0(A)/{\rm ker}\rho_A\not\cong \Z$ has tracial rank one or zero.
\end{thm}

\begin{proof}
Since $K_1(A)$ is a countable abelian group, we can write $$K_1(A)=\varinjlim_{n\to\infty}( G_{n}, \iota_{n}),$$ with $G_{n}\cong  \bigoplus_{l=1}^{l_n}\mathbb{Z}/p_{n, l}\mathbb{Z}$ for some non-negative integers $l_n$ and $p_{n, l}$, where $p_{n, l}\neq 1$ (note that if $p_{n, l}=0$, one has that $\mathbb{Z}/p_{n, l}\mathbb Z\cong\mathbb Z$). Denote by $p_n=\sum_{l=1}^{l_n}p_{n, l}$.

Since $K_0(A)$ is a simple Riesz group and $K_0(A)/{\rm ker}\rho_A\not\cong \Z$, and the pairing between $T(A)$ and $K_0(A)$ preserves extreme points, by \cite{Vill-EHS0}, there is a simple inductive limit of interval algebras $B$ such that
\begin{equation}\label{kzero}
((K_0(A), K_0(A)_+, [1_A]), T(A), \lambda_A)\cong ((K_0(B), K_0(B)_+, [1_B]), T(B), \lambda_B).
\end{equation}

Write $$B=\varinjlim_{n\to\infty} (\bigoplus_{i=1}^{k_n} B_{n, i}, \phi_{n}),$$ where $B_{n, i}=\mathrm{M}_{m_{n, i}}(\mathrm{C}([0, 1]))$, and write $[\phi_{n}]_0=(r_{n, i, j})$ with $r_{n, i, j}\in\mathbb{Z}^+$, where ${1\leq i\leq k_{n+1}}$ and $1\leq j\leq k_{n}$. Since $A$ is simple, without loss of generality, we may assume that $r_{n, 1, j}> (n+1)p_n$ by passing to a subsequence.

For each $n$ and each $1\leq j\leq k_{n+1}$, consider the restriction of the map $\phi_{n}$ to the direct summand $B_{n, 1}$ and $B_{n+1, j}$, that is, consider the map $\phi_{n}(1, j): B_{n, 1}\to B_{n+1, j}$. It follows from
\cite{Ell-AI} that we may assume that there exist continuous functions $s_1, s_2,..., s_{r_{n, 1, j}}$ such that
\begin{displaymath}
\phi_{n}(1, j)(f)(t)=W^*(t)\mathrm{diag}\{f\circ s_1(t),...,f\circ s_{r_{n, 1, j}}(t)\}W(t)
\end{displaymath}
for some $W(t)\in\mathrm{M}_{m_{n, i}r_{n, 1, j}}(\mathrm{C}([0, 1]))$. Factor through the map $\phi_{n}(1, j)$ by
\begin{displaymath}
\xymatrix{
B_{n, 1}\ar[r]^-{\psi_{1}}&(\bigoplus_{l=1}^{l_n}\mathrm{M}_{p_{n, l}m_{n, 1}}(\mathrm{C}([0, 1])))\oplus B_{n, 1} \ar[r]^-{\psi_2}& B_{n+1, j},
}
\end{displaymath}
where $$\psi_1(f)(t)=\mathrm{diag}\{\{f\circ s_1,...,f\circ s_{p_{n, 1}}\},..., \{f\circ s_{p_n-p_{n, l_n}+1},...,f\circ s_{p_n}\}, f\}.$$ and $$\psi_2(f\oplus g)=W^*\mathrm{diag}\{f,  g\circ s_{p_n+1},..., g\circ s_{r_{n, 1, j}}\}W$$ is the diagonal embedding. Since $r_{n, 1, j}\geq np_n$, one has that the restriction of any tracial state of $B$ to the unit of $\bigoplus_{l=1}^{l_n}\mathrm{M}_{p_{n, l}m_{n, 1}}(\mathrm{C}([0, 1]))$ has value less than $1/n$.

Therefore, by replacing $B_{n, 1}$ by $\mathrm{M}_{p_{n, l}m_{n, 1}}(\mathrm{C}([0, 1]))\oplus (\bigoplus_{r_{n, 1, j}-p_n}B_{n, 1})$, one may assume that there is an inductive limit decomposition $$B=\varinjlim_{n\to\infty} (\bigoplus_{i=1}^{k_n} B_{n, i}, \phi_{n}),$$ where $B_{n, i}=\mathrm{M}_{m_{n, i}}(\mathrm{C}([0, 1]))$, such that $k_n>l_n$, and $m_{n, j}=d_{n, j}p_{n, j}$ for some natural number $d_{n, j}$ for any $1\leq j\leq l_n$. Moreover, one has that $\tau(e)<1/n$ for any $\tau\in T(A)$, where $e$ is the unit of $\bigoplus_{j=1}^{l_n} B_{n, j}$. In other words,
\begin{equation}\label{smalltrace}
\rho(e)<1/n
\end{equation}
for any $\rho\in S_u(K_0(A))$. Fix this inductive limit decomposition.

Now, let us replace certain interval algebras at each level $n$ by certain dimension drop interval algebras so that the new inductive limit gives the desired $K_1$-group, and keep the $K_0$-group and the paring unchanged.

At level $n$, for each $1\leq l\leq l_n$, if $p_{n, l}\neq 0$, denote by $D_{n, l}$ the dimension drop C*-algebra $\mathbf{I}[m_{n, l}, m_{n, l}p_{n, l}, m_{n, l}]$; if $p_{n, l}=0$, denote by $D_{n, l}$ the circle algebra $\mathrm{M}_{m_{n, l}}(\mathrm{C}(\mathbb T))$. Then, one has $$(K_0(D_{n, l}), K^+(D_{n, l}), 1_{D_{n, l}})=(\mathbb Z, \mathbb Z^+, m_{n, l})\quad{\mathrm{and}}\quad K_1(D_{n, l})=\mathbb{Z}/p_{n, l}\mathbb{Z}.$$

Replace each $B_{n, l}$ by $D_{n, l}$, and denote by $$D_n=D_{n, 1}\oplus\cdots\oplus D_{n, l_n}\oplus B_{n, l_n+1}\oplus\cdots\oplus B_{n, k_n}.$$ It is clear that $$K_0(D_n)\cong K_0(B_n)\quad\mathrm{and}\quad K_1(D_n)\cong\bigoplus_{l=1}^{l_n}\mathbb{Z}/p_{n, l}\mathbb{Z}\cong G_n$$ Let us construct maps $\chi_{n}: D_n\to D_{n+1}$ as the following.

For the direct summand $D_{n, i}$ and any direct summand $D_{n+1, j}$, if $p_{n, i}$ and $p_{n+1, j}$ are non-zero, by Corollary 3.9 of \cite{JS-Z}, there is a map $\chi_{n}(i, j): D_{n+1, i}\to D_{n, j}$ such that $$[\chi_{n}(i, j)]_0=[\phi_{n}(i, j)]_0=r_{n, i, j}$$ and $$[\chi_{n}(i, j)]_1=\iota_n(i, j).$$

If $p_{n, i}=0$ and $p_{n+1, j}\neq 0$, then define the map $$\chi_{n}(i, j): \mathrm{M}_{m_{n, i}}(\mathbb C)\otimes(\mathrm{C}(\mathbb T))\cong D_{n, i}\to pD_{n+1, j}p\cong \mathrm{M}_{r_{n, i, j}m_{n, i}}(\mathbb C)\otimes\mathbf{I}[1, p_{n+1, j}, 1]$$ where $p$ is a projection stands for $r_{n, i, j}m_{n, i}\in K_0(D_{n+1, j})$, by $$e\otimes z\to\mathrm{diag}\{e\otimes u, e\otimes z_1,...,e\otimes z_{r_{n, i, j}-1}\},$$ where $z$ is the standard unitary $z\mapsto z$, $z_i$ are certain points in the unit circle, and $u$ is a unitary in $pD_{n+1, j}p$ which represents $\iota_n(i, j)(1)$. Then, it is clear that $$[\chi_{n}(i, j)]_0=[\phi_{n}(i, j)]_0=r_{n, i, j}$$ and $$[\chi_{n}(i, j)]_1=\iota_n(i, j).$$

A similar argument for $p_{n, i}\neq 0$ and $p_{n+1, j}=0$ also provides a homomorphism $\chi_n(i, j)$ which induces the right $K$-theory map. Moreover, the argument above also applies to the maps between $D_{n, i}$ and $B_{n+1, j}$, and between $B_{n, i}$ and $D_{n+1, j}$, such that there is a map $\chi_{n}(i, j)$ with $$[\chi_{n}(i, j)]_0=[\phi_{n}(i, j)]_0=r_{n, i, j}$$ and $$[\chi_{n}(i, j)]_1=\iota_n(i, j).$$

For direct summand $B_{n, i}$ and $B_{n+1, j}$, define $$\chi_{n}(i, j)=\phi_{n}(i, j).$$

In this way, we have a homomorphism $\chi_n: D_n\to D_{n+1}$ satisfying $$[\chi_n]_0=[\phi_n]\quad\mathrm{and}\quad [\chi_n]_1=\iota_{n, n+1}$$
Let us consider the inductive limit $$D=\varinjlim_{n\to\infty}(D_n, \chi_n).$$ It is clear that $K_0(D)=K_0(B)$ and $K_1(D)=\varinjlim(G_n, \iota_n)=K_0(A)$. With a suitable choice of $\chi_{n}(i, j)$ between $D_{n, i}$ and $D_{n+1, j}$, $D_{n, i}$ and $B_{n+1, j}$, and $B_{n, i}$ and $D_{n+1, j}$, we may assume that $D$ is a simple C*-algebra. Let us show $\mathrm{TR}(D)\leq 1$.

From the construction, it is clear that $D$ has the following property: For any finite subset $\mathcal F\subset D$ and any $\ep>0$, there exists $n$ such that if denote by $I_n=\bigoplus_{i=l_n+1}^{p_n} B_{n, i}$ and $p_n=1_{I_n}$, then, for any $x\in \mathcal F$
\begin{enumerate}
\item $||[p_n, x]||\leq\ep$, and
\item there is $a\in I_n$ such that $||{p_nxp_n-a}||\leq\ep$.
\end{enumerate}
By Theorem 3.2 of \cite{LinTAI}, the C*-algebra has the property (SP), that is, any nonzero hereditary sub-C*-algebra contains a nonzero projection. Thus, in order to show $\mathrm{TR}(D)\leq 1$, one only has to show that for any given projection $q\in D$, one can choose $n$ sufficiently large such that $1-p_n\lesssim q$.

Note that the C*-algebra $D$ is an inductive limit of dimension drop interval algebras together with circle algebras, which satisfy the strict comparison on projections, i.e., for any two projections $e$ and $f$, if $\tau(e)<\tau(f)$ for any tracial state $\tau$, then $e\lesssim f$. Then $D$ also has the strict comparison on projections. (See, for example, Theorem 4.12 of \cite{EN-Tapprox}.)

Therefore, in order to show $1-p_n\lesssim q$, one only has to show that for any given $\ep>0$, there is a sufficiently large $n$ such that $\tau(1-p_n)\leq\ep$ for any $t\in T(D)$. However, this condition can be fulfilled by Equation \ref{smalltrace}, and thus, the C*-algebra $D$ is tracial rank one.

Let us show that $B$ and $D$ has the same tracial simplex, and has the same pairing with the $K_0$-group. Consider the non-unital C*-algebra $$C=\varinjlim_{n\to\infty}(\bigoplus_{i=l_n+1}^{k_n} B_{n, i}, \psi_n),$$ where the map $\psi_n$ is the restriction of $\phi_n$ to $\bigoplus_{i=l_n+1}^{k_n} B_{n, i}$ and $\bigoplus_{i=l_{n+1}+1}^{k_{n+1}} B_{n+1, i}$. Then, by Lemma 10.8 (and its proof) of \cite{LinTAI}, there are isomorphisms ${r^{\#}}$ and ${r_{\#}}$, and ${s^{\#}}$ and ${s_{\#}}$ such that
\begin{displaymath}
\xymatrix{
T(B)\ar[r]^-{r_B}\ar[d]_-{r^{\#}} & S_u(K_0(B))\\
T(C)\ar[r]^-{r_C} & S_{u'}(K_0(C))\ar[u]_-{r_{\#}},
}\quad\quad
\xymatrix{
T(D)\ar[r]^-{r_D}\ar[d]_-{s^{\#}} & S_u(K_0(D))\\
T(C)\ar[r]^-{r_C} & S_{u'}(K_0(C))\ar[u]_-{s_{\#}}
}
\end{displaymath}
commutes. Therefore, there are isomorphisms $t^\#$ and $t_\#$ such that
\begin{displaymath}
\xymatrix{
T(B)\ar[r]^-{r_B}\ar[d]_-{t^{\#}} & S_u(K_0(B))\\
T(D)\ar[r]^-{r_D} & S_{u'}(K_0(D))\ar[u]_-{t_{\#}}
}
\end{displaymath}
commutes. Therefore, the C*-algebra $B$ and $D$ has the same simplex of traces and pairing map. Hence, $$((K_0(D), K_0(D)_+, [1_D]), T(D), \lambda_A)\cong ((K_0(B), K_0(B)_+, [1_B]), T(B), \lambda_B),$$ and therefore by Equation \ref{kzero}, $$((K_0(D), K_0(D)_+, [1_D]), K_1(D), T(D), \lambda_A)\cong ((K_0(A), K_0(A)_+, [1_A]), K_1(A) T(A), \lambda_A).$$

Since $D$ is also an inductive limit of dimension drop interval algebras together with circle algebras, by (the proof of) Theorem 9.9 of \cite{Myg-CJM}, one has that $D$ is a simple inductive limit of dimension drop circle algebra, and hence by Theorem 11.7 of \cite{Myg-CJM}, one has that $A\cong D$. Since $TR(D)\leq 1$, we have that $TR(A)\leq 1$, as desired.
\end{proof}

\begin{thm}\label {AD2}
Every unital simple A$\T$D-algebra is in ${\cal A}_{\cal Z}.$
\end{thm}

\begin{proof}
If $K_0(A)/{\rm Inf}(K_0(A))\not\cong \Z,$ then, \ref{AD1} shows that $TR(A)\le 1.$ By Theorem 10.4 of \cite{LinTAI},
$A$ is in fact a unital simple AH-algebra with no dimension growth. It follows from \cite{EGL-ADiv} that $A$ is also
approximately divisible. It follows from Theorem 2.3 of \cite{TW1} that $A$ is ${\cal Z}$-stable. So $A\in {\cal A}_{\cal Z}.$

Let $A$ be a general unital simple A$\T$D-algebra. Let ${\mathfrak{p}}$ be a supernatural number of infinite type.
Then $K_0(A\otimes M_{\mathfrak{p}})/{\rm ker}\rho_A\not \cong \Z.$ It follows from \ref{AD1} that
$TR(A\otimes M_\fp)\le 1.$ Thus $A\in {\cal A}.$

It follows from Theorem 4.5 of \cite{TW1} that $A$ is ${\cal Z}$-stable. Therefore $A\in {\cal A}_{\cal Z}.$
\end{proof}

\begin{NN}\label{NN1}
{\rm
From Theorem \ref{AD2} we see that Theorem \ref{CT0} also unifies the classification
theorems of \cite{LinTAI} and that of \cite{Myg-CJM}. In the next two sessions, we will show that
${\cal A}$ contains many more \CA s.

}
\end{NN}

\section{Rationally Riesz Groups}

\begin{thm}\label{RZ} For any countable abelian groups $G_{00}$ and $G_1,$ any group
extension:
$$
0\to G_{00}\to G_0 {\stackrel{\pi}{\to}} \Z\to 0,
$$
with
$$
(G_0)_+=\{x\in G_0: \pi(x)>0\andeqn x=0\},
$$
any order unit $u\in G_0$
and any metrizable Choquet simple $S,$ there exists a unital simple ASH-algebra $A\in {\cal A}$ such that
$$
((K_0(A), K_0(A)_+, [1_A]), K_1(A), T(A))=((G_0, (G_0)_+, u), G_1, S).
$$
\end{thm}

\begin{proof}
Let $S_0$ be the point which corresponding
to the unique state on $G_0.$
It follows a theorem of Elliott (\cite{point-line}) that there exists a unital simple ASH-algebra $B$ such that
$$
((K_0(B), K_0(B)_+, [1_B]), K_1(B), T(B))=((G_0, (G_0)_+, u), G_1, S_0).
$$
%By replacing $B$ by $B\otimes {\cal Z},$ we may assume that $B$ is ${\cal Z}$-stable.
Then $B\otimes M_\fp$ is a unital separable simple \CA\, which is
approximately divisible and the projections of $B\otimes M_\fp$
separate the tracial state space (in this case it contains a single
point). Thus $B\otimes M_\fp$ has real rank zero and stable rank one
by \cite{RorUHF}. Thus, by  Proposition 5.4 of \cite{Lin-Corelle}, $TR(B\otimes
M_\fp)=0.$ Let $B_0$ be the unital simple A$\T$D-algebra  (see
Theorem 4.5 of \cite{JS-Z}) such that
$$
((K_0(B_0), K_0(B)_+, [1_{B_0}]), K_1(B_0), T(B_0), \lambda_{B_0})=
((\Z, \N, 1), \{0\}, S, r_S).
$$
It follows from \ref{AD2} that $B_0\in {\cal A}.$ By Theorem 11.10 (iv) of \cite{Lnclasn}, $B_0\otimes B\in {\cal A}.$
Define $A=B_0\otimes B.$ One then calculates that
$$
((K_0(A), K_0(A)_+, [1_A]), K_1(A), T(A))=((G_0, (G_0)_+, u), G_1, S).
$$

\end{proof}

\begin{thm}\label{R2} For any countable weakly unperforated simple Riesz group $G_0$ with order unit  $u,$ any
countable abelian group $G_1$ and any metrizable Choquet simplex
$S$ and any surjective \hm\,  $r_S: S\to S_u(G_0)$ which maps
$\partial_e(S)$ onto $\partial_e(S_u(G_0)).$ There exists a unital
simple ASH-algebra $A\in {\cal A}_z$ such that
$$
((K_0(A), K_0(A)_+, [1_A]), K_1(A), T(A), \lambda_A)=((G_0, (G_0)_+, u), G_1, S, r_S).
$$
\end{thm}

\begin{proof}
First, we assume that $G_0/{\rm Inf}(G_0)\not\cong \Z.$ It follows from
a theorem of Villadsen (\cite{Vill-EHS}) that, in this case,  there is a
unital simple AH-algebra $C$ with no dimension growth (and with
tracial rank no more than one---see Theorem 2.5 of \cite{LinTAI}) such that
$$
((K_0(C),K_0(C)_+, [1_C]), K_1(C), T(C), \lambda_C)=((G_0, (G_0)_+, u), G_1, S, r_S).
$$
The case that $G_0/{\rm Inf}(G_0)\cong \Z$ follows from \ref{RZ}.
\end{proof}

\begin{rem}\label{rem1}
{\rm
Theorem \ref{R2} includes all unital simple A$\T$D-algebras. Let $A$ be a unital simple \CA\, in ${\cal A}$
with weakly unperforated Riesz group $K_0(A).$  If $K_0(A)/{\rm Inf}(K_0(A))\not\cong \Z,$ then
$TR(A)\le1 .$ By the classification result in \cite{LinTAI}, $A$ is in fact a unital simple AH-algebra.
If $K_0(A)=\Z,$ then, by \ref{AD2}, $A$ is a unital simple A$\T$D-algebra.
However,
Theorem \ref{R2} contains unital simple \CA \, $A$ for which
$K_1(A)$ is an arbitrary countable abelian group.  These
\CA s can not be isomorphic to a unital simple A$\T$D-algebras (see Theorem 1.4 of \cite{Myg-CJM}).

}
\end{rem}

\begin{df}\label{dfRi}
{\rm
Let $G$ be a partially ordered group.
We say $G$ has rationally  Riesz property if the following holds:
For two pairs of elements $x_1, x_2,$ and $y_1, y_2\in G$ with $x_i\le y_j,$ $i,j=1,2,$ there exists
$z\in G$ such that
$$
nz\le m y_i\andeqn mx_i\le nz,\,\,\,i=1,2,.
$$
where $m,n \in \N\setminus\{0\}.$
}
\end{df}

Let $G$ and $H$ be two weakly unperforated simple ordered groups with order-unit $u$ and $v$ respectively.
Consider the group $G\otimes H$. Set the semigroup
$$(G\otimes H)_+=\{a\in G\otimes H;\ (s_1\otimes s_2)(a)>0, \forall s_1\in S_u(G), \forall s_2\in S_v(H)\}\cup\{0\}.$$
Since $G_+\otimes H_+\subseteq (G\otimes H)_+$ and  $G\otimes
H=G_+\otimes H_+-G_+\otimes H_+$, one has $$G\otimes H=(G\otimes
H)_+-(G\otimes H)_+.$$ For any $a\in((G\otimes H)_+)\cap(-(G\otimes
H)_+)$, if $a\neq 0$, then $(s_1\otimes s_2)(a)>0$ for any $s_1\in
S_u(G)$ and  $s_2\in S_v(H)$, and $(s_1\otimes s_2)(-a)>0$, for any
$s_1\in S_u(G)$ and $s_2\in S_v(H)$, which is a contradiction.
Therefore, $$((G\otimes H)_+)\cap(-(G\otimes H)_+)=\{0\}.$$
Moreover, since $(s_1\otimes s_2)(u\otimes v)=1$ for any $s_1\in
S_u(G)$ and $s_2\in S_v(H)$, and $S_u(G)$ and $S_v(H)$ are compact,
for any element $a\in (G\otimes H)$, there is a natural number $m$
such that $$m(u\otimes v)-a\in (G\otimes H)_+.$$ Hence $(G\otimes H,
(G\otimes H)_+, u\otimes v)$ is a scaled ordered group.

\begin{lem}\label{tensorstate}
Let $G$ and $H$ be simple ordered groups with order units $u$ and $v$ respectively.
If $H$ has a unique state $\tau$, then for any state $s$ on the ordered group $G\otimes H$,
one has $$s(g\otimes h)=s(g\otimes v)\tau(h)$$ for any $g\in G$, $h\in H$.
\end{lem}
\begin{proof}
It is enough to show the statement for strictly positive $g$ and
$h$. For each $g\in G_+\setminus\{0\}$, since $G$ is simple, one has
that $u\leq mg$ for some natural number $m$. Hence $$s(g\otimes
v)\geq \frac{1}{m}s(u\otimes v)=\frac{1}{m},$$ and in particular,
$s(g\otimes v)\neq0$. Consider the map $s_g: H\to \R$ defined by
$$
s_g(h)=\frac{s(g\otimes h)}{s(g\otimes v)}.
$$
Then $s_g$
is a state of $H$. Since $\tau$ is the unique state of $H$. One has
that $s_g=\tau$, and hence $$\frac{s(g\otimes h)}{s(g\otimes
v)}=\tau(h)\quad\textrm{for any}\ h\in H.$$ Therefore, $$s(g\otimes
h)=s(g\otimes v)\tau(h),$$ as desired.
\end{proof}

\begin{lem}\label{Rstate}
Let $G$ be a countable weakly unperforated simple partially ordered group with an order unit $u.$
Then,  for any dense subgroup $D$ of $\R$ containing 1,   the map
$\lambda: S_{u\otimes 1}(G\otimes D)\to S_u(G)$ defined by
$$
\lambda(s)(x)=s(j(x))\tforal x\in G,
$$
where $j: G\to G\otimes D$ defined by $j(x)=x\otimes 1$ for all
$x,$ is an affine homeomorphism.
\end{lem}

\begin{proof}
Since any dense subgroup of $\R$ with the induced order has unique state, the statement follows from Lemma \ref{tensorstate} directly.
%Clearly that $\lambda$ is an affine continuous map.  Suppose that
%$s_1, s_2\in S_{u\otimes 1}(G)$ such that
%$\lambda(s_1)=\lambda(s_2).$ In other words,
%$$
%s_1(j(x))=s_2(j(x))\rforal x\in G.
%$$
%For any $r\in \D,$ one must have
%$$
%s_1(x\otimes r)=s_1(rj(x))=rs_1(x)=rs_2(x)=s_2(x\otimes r).
%$$
%It follows that $s_1=s_2,$ whence $\lambda$ is one-to-one.
%
%Note that for any $x\in {\rm ker}j,$ $mx=0$ for some $m\in
%\N\setminus\{0\}.$ It follows that $x\in Inf(G).$ Therefore
%$S_u(G)=S_{{\bar u}}(G/{\rm ker}j),$ where ${\bar u}$ be the image
%of $u$ in $G/{\rm ker} j.$ Thus, we may replace  $G$ by $G/{\rm
%ker}j$ and assume that $j$ is injective which follows that
%$\lambda$ is surjective.
%
\end{proof}

\begin{prop}\label{PRi}
Let $G$ be a countable weakly unperforated simple partially ordered group with an order unit $u.$
Then the following are equivalent:

{\rm (1)} $G$ has  the rationally  Riesz property;

{\rm (2)} $G\otimes D$ has the Riesz property for some dense subgroup $D$ of $\R$ containing $1;$

{\rm (3)} $G\otimes D$ has  the Riesz property for all dense subgroups $D$ of $\R$ containing $1;$

{\rm (4}) For any two pairs of elements $x_i$ and $y_i\in G$ with $x_i\le y_j,$ $i,j=1,2,$ there is an element $z\in G$ and
a real number $r>0$ such that
$$
s(x_i)\leq r s(z)\leq s(y_j),\tforal s\in S_u(G),\,i,j=1,2.
$$
\end{prop}

\begin{proof}

It is clear, by the assumption that $G$ is weakly unperforated, that
{\rm (1)} $\Longrightarrow$ {\rm (2)} for $D=\Q.$

It is also clear that  {\rm (3)} $\Longrightarrow$ {\rm (2)}.

That {\rm (2)} $\Longrightarrow$ {\rm (4)} follows from \ref{Rstate}.

Suppose that {\rm (4)} holds. Let $x_i\le y_j$ be in $G,$ $i, j=1,2.$ If one of $x_i$ is the same as one of $y_j$, say $x_1=y_1$,
then {\rm (2)} holds for $z=x_2$ and $m=n=1$. Thus, let us assume that $x_i<y_j$, $i, j=1, 2$.
Since $G$ is simple, $S_u(G)$ is compact. It follows from {\rm (4)} there is a
rational number $r\in \Q$ such that
$$
s(x_i)<r s(z)<s(y_j),\tforal s\in S_u(G),\,i,j=1,2.
$$
Write $r=n/m$ for some $m,n\in \N.$ Then
$$
s(mx_i)<s(nz)<s(my_j)\rforal s\in S_u(G),\,i,j=1,2.
$$
It follows from Theorem 6.8.5 of \cite{Bla-Ktheory} that
$$
mx_i\le nz\le my_j\,\,\, i,j=1,2.
$$

Thus (4)  $\Longrightarrow$ (1).

By applying \ref{Rstate}, it is even easier to shows that {\rm
(4)} $\Longrightarrow$ {\rm (3)}.
\end{proof}

\begin{prop}\label{PRi2}
Let $G$ be a countable weakly unperforated simple partially ordered
group with an order unit $u.$ Then $G$ has { the} rationally
Riesz property if and only if $S_u(G)$ is a metrizable Choquet
simplex.
\end{prop}

\begin{proof}
Suppose that $G$ has the rationally  Riesz property. Then, by \ref{PRi},
$G\otimes \Q$ is a weakly unperforated simple Riesz group and
$(G\otimes \Q)/{\rm Inf}(G\otimes \Q)\not\cong \Z.$ Put
$F=(G\otimes \Q)/{\rm Inf}(G\otimes \Q).$ Then, it follows that
$F$ is a simple dimension group. It then follows from the Effros-Handelman-Shen Theorem (\cite{EHS}) that there
exists a unital simple AF-algebra $A$ with
$$
(K_0(A), K_0(A), [1_A])=(F, F_+, {\bar u}),
$$
where ${\bar u}$ is the image of $u$ in $F.$ It follows that
$T(A)=S_u(F).$ By Theorem 3.1.18 of \cite{Sakai-book}, $T(A)$ is a metrizable Choquet simplex. It
follows from \ref{Rstate} that $S_u(G)$ is a metrizable Choquet
simplex.

For the converse, let $G$ be a countable weakly unperforated
simple partially ordered group with an order unit $u$ so that
$S_u(G)$ is a metrizable Choquet simplex. Let $F=G\otimes \Q.$ By
\ref{PRi}, it suffices to show that $F$ has the Riesz property. By
\ref{Rstate}, $S_{u\otimes 1}(F)$ is a metrizable Choquet simplex.
It follows from Theorem 11.4 of \cite{Goodearl} that $\mathrm{Aff}(S_{u\otimes 1}(F))$ has the
Riesz property. Let $\rho: F\to \mathrm{Aff}(S_{u\otimes 1}(F))$ be the
\hm\, defined by
$$
\rho(g)(s)=s(g)\tforal s\in S_{u\otimes 1}(G)\andeqn \tforal g\in
F.
$$
Define $F_1=F+\R(u\otimes 1)$ and extend $\rho$ from $F_1$ into
$\mathrm{Aff}(S_{u\otimes 1}(F))$ in an obvious way.  Then
$\rho(F_1)$ contains the constant functions. It also separates the
points. By Corollary 7.4 of \cite{Goodearl}, the linear space generated by $\rho(F_1)$ is dense in
$\mathrm{Aff}(S_{u\otimes 1}(G)).$

Moreover, for any real number $r<1$ and any positive element $p\in
F$, the positive affine function $r\rho(p)$ can be approximated by
elements in $\rho(F)$. It follows that $\rho(F)$ is dense in
$\mathrm{Aff}(S_{u\otimes 1}(G)).$ It follows that $\rho(F)$ has the
Riesz property. Moreover, since the order of $F$ is given by
$\rho(F)$ (see Theorem 6.8.5 of \cite{Bla-Ktheory}), this implies that $F$ has the Riesz property.
\end{proof}

\begin{exm}
Let $H$ be any nontrivial group. Then ordered group $\mathbb Z\oplus H$ with the order induced by $\mathbb Z$ is a
simple ordered group, which satisfies the rationally  Riesz property. However, it is not a Riesz group.

Let $\Gamma$ be a cardinality at most countable and bigger than $1$.
Then, the ordered group $G=\bigoplus_{\Gamma}\mathbb Z$ with the
positive cone $\{0\}\cup\bigoplus_{\Gamma}\mathbb Z^+$ is simple,
since any positive element is an order unit. Consider $a_1=(1, 0, 0,
0,...)$, $a_2=(0, 1, 0, 0, ...)$, $a_3=(2, 2, 0, 0, ...)$, and $
a_4=(2, 3, 0, 0, ...)$. Then $a_1, a_2\leq a_3, a_4$. However, one
can not find an element $b$ such that $a_1, a_2\leq b\leq a_3, a_4$.
Hence, the group $G$ is not Riesz. But this group has
rationally Riesz property.
\end{exm}

\begin{prop}\label{R0}
Let $A\in {\cal A}$ be a $\mathcal Z$-stable  \CA. Then $(K_0(A), K_0(A)_+,
[1_A])$ is a countable weakly unperforated simple partially
ordered group with order unit $[1_A]$ which has the rationally  Riesz
property. Moreover $S_{[1_A]}(K_0(A))$ is a metrizable Choquet
simplex.
\end{prop}

\begin{proof}

It follows from \cite{GJS-Z} that $(K_0(A), K_0(A)_+, [1_A])$ is a
countable weakly unperforated simple partially ordered group with
order unit $[1_A].$ Since $TR(A\otimes \Q)\le 1,$ $(K_0(A\otimes
\Q), K_0(A\otimes \Q)_+, [1_{A\otimes \Q}])$ is a Riesz group. It
follows that $(K_0(A), K_0(A)_+, [1_A])$ has the rationally  Riesz
property. By \ref{Rstate} and \ref{PRi2}, $S_{[1_A]}(K_0(A))$ is a
metrizable Choquet simplex.

\end{proof}

\begin{lem}\label{LRa} Let $G_0$ be a countable weakly unperforated simple
partially ordered group with order unit $u$ which also has the
rationally  Riesz property and any countable abelian group $G_1.$ There
exists a unital simple ASH-algebra $A\in {\cal A}_{0z}\subset
{\cal A}_z$ such that
$$
((K_0(A), K_0(A)_+, [1_A]), K_1(A), T(A))=((G_0, (G_0)_+, u), G_1, S_u(G_0)).
$$
\end{lem}

\begin{proof}
It follows from \ref{PRi2} that $S_u(G_0)$ is a Choquet simplex.
It follows \cite{point-line} that there exists a unital simple
ASH-algebra $A$ such that
$$
((K_0(A), K_0(A)_+, [1_A]), K_1(A), T(A))=((G_0, (G_0)_+, u), G_1, S_u(G_0)).
$$
We may assume that $A\cong A\otimes {\cal Z}.$ Note that the set of projections of $A\otimes M_\fp$
separate the tracial state space in this case for any supernatural number $\fp.$
It follows from the argument of 8.2 of \cite{Winter-Z} that $TR(A\otimes M_\fp)=0.$ In particular, $A\in {\cal A}.$
\end{proof}

\begin{df}\label{span}
Let $T_1$ and $T_2$ be two finite simplexes with vertices $\{e_1, ..., e_m\}$ and $\{f_1, ..., f_n\}$. Let us denote by $T_1\dot\times T_2$ the finite simplex spanned by vertices $(e_i, f_j)$, $1\leq i\leq m$, $1\leq j\leq n$.
\end{df}

\begin{lem}\label{rep}
Let $A$ be unital C*-algebra and $\tau$ a tracial state of $A$. Then
$\tau$ is extremal if and only if $\pi_\tau(A)''$ in the \textcyr{GN}{$\mathrm S$}-representation $(\pi_\tau, \mathscr H_\tau)$ is a
$\mathrm{II}_1$ factor.
\end{lem}
\begin{proof}
Note that $\pi_\tau(A)''$ is always of type $\textrm{II}_1$. Assume that $\tau$
is extremal. If $\pi_\tau(A)''$ were not a factor, then, there is a nontrivial central
projection $p\in \pi_\tau(A)''$.

We claim that $\tau(p)\not=0.$ Suppose that $a_n\in A_+$ such that
$\{\|\pi_\tau(a_n)\|\}$ is bounded and $\pi_\tau(a_n)$ converges to
$p$ in the weak operator topology in $B({\mathscr{H}}_\tau).$  If
$\tau(p)=0,$ then $\tau(a_n)\to 0.$ It follows that $\tau(a_n^2)\to
0.$ For any $a\in A\setminus\{0\},$
$$
\tau(aa_na^*)=\tau(a_na^*a)\le (\tau(a_n^2)\tau((a^*a)^2))^{1/2}\to
0.
$$
It follows that
$$
\lim_{n\to\infty}\langle\pi_\tau(a_n)\xi_a, \xi_a\rangle_{\mathscr{H}_\tau}=0
$$
for any $a\in A,$ where $\xi_a$ is the vector given by $a$ in the
\textcyr{GN}S construction, which implies that $\pi_\tau(a_n)\to 0$ in
the weak operator topology. Therefore $p=0.$ This proves the claim.

By the claim neither  $\tau(p)$ nor $\tau(1-p)$ is zero. Thus
$0<\tau(p)<1.$ Define $\tau_1: a\mapsto\frac{1}{\tau(p)}\tau(pap)$
and $\tau_2: a\mapsto\frac{1}{\tau(1-p)}\tau((1-p)a(1-p))$. Note
that $\tau_1(p)=1.$ Thus $\tau_1\not=\tau.$

Then $\tau_1$ and $\tau_2$ are traces on $\pi_\tau(A)''$, and
$\tau(a)=\tau(p)\tau_1(a)+\tau(1-p)\tau_2(a)$ for any $a\in
\pi_\tau(A)''$. Therefore, the trace $\tau$ on $\pi_\tau(A)''$ is
not extremal. Since both $\tau_1$ and $\tau_2$ are also normal,  we
conclude that  $\tau$ is also not extremal on $A$, which contradicts
to the assumption.

Conversely, assume that $\pi_\tau(A)''$ is a factor. If
$\tau=\lambda\tau_1+(1-\lambda)\tau_2$ with $\lambda\in(0, 1)$, then
it is easy to check  that $\tau_1$ and $\tau_2$ can be extended to
normal states on $\pi_\tau(A)''$, and hence to traces on
$\pi_\tau(A)''$. Therefore $\tau_1=\tau_2$, and $\tau$ is extremal.
\end{proof}

\begin{lem}\label{restriction}
Let $A$ and $B$ be two unital C*-algebras. Let $\tau$ be an extremal tracial state on a C*-algebra tensor product $A\otimes B$. Then, the restriction of $\tau$ to $A$ or $B$ is an extremal trace.
\end{lem}
\begin{proof}
Denote by the restriction of $\tau$ to $A$ and $B$ by $\tau_A$ and
$\tau_B$. Consider the \textcyr{GN}S-representation $(\pi_\tau,
\mathscr H_\tau)$ of $A\otimes B$. Since $\tau$ is extremal, by
Lemma \ref{rep}, the von Neumann algebra $\pi_\tau(A\otimes B)''$ is
a $\textrm{II}_1$ factor. Since $(A\otimes 1)$ commutes with
$(1\otimes B)$, $\pi_\tau(A\otimes 1)''$ and $\pi_\tau(1\otimes
B)''$ are also $\textrm{II}_1$ factors.

Set $p$ the orthogonal projection to the closure of the subspace
spanned by $\{(a\otimes 1)(\xi); \xi\in\mathscr H_\tau, a\in A\}$.
Then the \textcyr{GN}S-representation $(\pi_{\tau_A}, \mathscr
H_{\tau_A})$ of $A$ is unitarily equivalent to the cut-down of
$\pi_\tau$ to $A$ and $p\mathscr H_\tau$. Hence $\pi_{\tau_A}(A)''$
is a $\textrm{II}_1$ factor in $\mathcal B(\mathscr H_{\tau_A})$. By
Lemma \ref{rep}, $\tau_A$ is an extremal trace of $A$. The same
argument works for $B$.
\end{proof}

\begin{lem}\label{prod0}
Let $A$ and $B$ be two unital C*-algebras and let $\tau$ be a
tracial state of a C*-algebra tensor product $A\otimes B$. Then, if
the restriction of $\tau$ to $B$ is an extremal trace, one has that
$\tau(a\otimes b)=\tau(a\otimes 1)\tau(1\otimes b)$ for all $a\in A$
and $b\in B$.
\end{lem}
\begin{proof}
We may assume that $a$ is a positive element with norm one. If $\tau(a\otimes 1)=0$, then $\tau(a\otimes b)=0$ by Cauchy-Schwartz  inequality, and equation holds. If $\tau(a\otimes 1)=1$, then the equation also hold by considering the element $(1-a)\otimes 1$.

Therefore, we may assume that $\tau(a\otimes 1)\neq 0, 1$. Fix $a$. Then, we have $$\tau(1\otimes b)=\tau(a\otimes 1)\frac{\tau(a\otimes b)}{\tau(a\otimes 1)}+(1-\tau(a\otimes 1))\frac{\tau((1-a)\otimes b)}{1-\tau(a\otimes 1)}\quad\textrm{for any}\ b\in B.$$ Note that both $b\mapsto\frac{\tau(a\otimes b)}{\tau(a\otimes 1)}$ and $b\mapsto\frac{\tau((1-a)\otimes b)}{1-\tau(a\otimes 1)}$ are tracial states of $B$. Since $b\mapsto \tau(1\otimes b)$ is an extremal trace, one has that $\frac{\tau(a\otimes b)}{\tau(a\otimes 1)}=\tau(1\otimes b)$ for any $b\in B$. Therefore, the equation $$\tau(a\otimes b)=\tau(a\otimes 1)\tau(1\otimes b)$$ holds for any $a\in A$ and $b\in B$.
\end{proof}

\begin{cor}\label{prod}
Let $A$ and $B$ be two unital C*-algebras and $\tau$ an extremal tracial state of a C*-algebra tensor product $A\otimes B$. Then $\tau$ is the product of its restrictions to $A$ and $B$
\end{cor}
\begin{proof}
It follows from Lemma \ref{restriction} and Lemma \ref{prod0}.
\end{proof}

\begin{cor}\label{trace-prod}
Let $A$ and $B$ be two C*-algebras with simplexes of traces
$\mathrm{T}(A)$ and $\mathrm{T}(B)$. If $\mathrm{T}(A)$ and
$\mathrm{T}(B)$ have finitely many extreme points, then $T(A\otimes
B)=\mathrm{T}(A)\dot\times \mathrm{T}(B)$.
\end{cor}
\begin{proof}
It follows from Corollary \ref{prod} directly.
\end{proof}

\begin{thm}\label{R3}
Let $G_0$ be a countable weakly unperforated simple partially
ordered group with an order unit $u$ which has the rationally  Riesz
property,  let $G_1$ be a countable abelian group,  and let $T$ be
any finite simplex. Assume that $S_u(G_0)$ has only finitely many extreme points. Then there exists a unital simple
ASH-algebra $A\in {\cal A}$ such that
$$
((K_0(A), K_0(A)_+, [1_A]), K_1(A), T(A), \lambda_A)=((G_0, (G_0)_+, u), G_1, T\dot{\times} S_u(G_0), r),
$$
where $r: T\dot{\times} S_u(G_0)\to S_u(G_0)$ is defined by $(\tau, s)(x)=s(x)$ for all $x\in G_0$ and for
all extremal tracial state $\tau\in T$ and extremal state $s\in S_u(G_0).$
\end{thm}

\begin{proof}

From \ref{LRa}, there exists a unital simple ASH-algebra $B\in {\cal A}$ such that
$$
((K_0(B), K_0(B)_+, [1_B]), K_1(B), T(B))=((G_0, (G_0)_+, u), G_1, S_u(G_0)).
$$
Then let $B_0$ be a unital simple A$\T$D-algebra with
$$
((K_0(B_0), K_0(B_0)_+, [1_{B_0}]), K_1(B_0), T(B_0))=((\Z, \N, 1), \{0\}, T).
$$
Define $A=B_0\otimes B.$ Then $A\in {\cal A}.$ One checks that
$$
((K_0(A), K_0(A)_+, [1_A]), K_1(A), T(A), \lambda_A)=((G_0, (G_0)_+, u, G_1), T\dot{\times} S_u(G_0), r).
$$
\end{proof}

\section{The range}

Suppose that $A$ and $B$ are two stably finite unital ${\cal Z}$-stable \CA s and suppose that
there is a \hm\,
$$
\Lambda: \mathrm{Ell}(A)\to \mathrm{Ell}(B).
$$
There is $\Lambda_\fp: \mathrm{Ell}(A\otimes M_{\fp})\to \mathrm{Ell}(B\otimes M_\fp)$ and
$\Lambda_\fq: \mathrm{Ell}(A\otimes M_\fp)\to \mathrm{Ell}(B\otimes M_\fq)$ induced by $\Lambda$ so that the following diagram commutes
$$
\begin{array}{ccccc}
\mathrm{Ell}(A\otimes M_\fp) & {\stackrel{({\rm id}_A\otimes 1)_*}{\leftarrow}}& \mathrm{Ell}(A) &{\stackrel{({\rm id}_A\otimes 1)_*}{\rightarrow}} & \mathrm{Ell}(A\otimes
M_\fq)\\
\downarrow_{\Lambda_\fp} & & \downarrow_{\Lambda} & &\downarrow_{\Lambda_\fq}\\
\mathrm{Ell}(B\otimes M_\fp) & {\stackrel{({\rm id}_B\otimes 1)_*}{\leftarrow}}& \mathrm{Ell}(B) &{\stackrel{({\rm id}_B\otimes 1)_*}{\rightarrow}} & \mathrm{Ell}(B\otimes
M_\fq)
\end{array}
$$

\begin{lem}\label{ext}
Let $A$ and $B$ be two ${\cal Z}$-stable \CA s in ${\cal A}$ and let $\fp$ and $\fq$ be two supernatural numbers of infinite type which are relatively prime. Suppose that
$$
\Lambda: \mathrm{Ell}(A)\to \mathrm{Ell}(B)
$$ is a \hm. Then there is a unitarily suspended $C([0,1])$-unital \hm s $\phi: A\otimes {\cal Z}_{\fp,\fq}\to B\otimes {\cal Z}_{\fp,\fq}$ such that
$Ell(\pi_0\circ \phi)=\Lambda_\fp$ and $Ell(\pi_1\circ \phi)=\Lambda_\fq$ so that the following diagram
 commutes:
$$
\begin{array}{ccccc}
\mathrm{Ell}(A\otimes M_\fp) & {\stackrel{({\rm id}_A\otimes 1)_*}{\leftarrow}}& \mathrm{Ell}(A) &{\stackrel{({\rm id}_A\otimes 1)_*}{\rightarrow}} & \mathrm{Ell}(A\otimes
M_\fq)\\
\downarrow_{(\pi_0\circ \phi)_*} & & \downarrow_{\Lambda} & &\downarrow_{(\pi_1\circ \phi)_*}\\
\mathrm{Ell}(B\otimes M_\fp) & {\stackrel{({\rm id}_B\otimes 1)_*}{\leftarrow}}& \mathrm{Ell}(B) &{\stackrel{({\rm id}_B\otimes 1)_*}{\rightarrow}} & \mathrm{Ell}(B\otimes
M_\fq)
\end{array}
$$
\end{lem}

\begin{proof}

Since $A,\, B\in {\cal A},$ $TR(A\otimes M_\fp)\le 1,$ $TR(A\otimes M_\fq)\le 1,$
$TR(B\otimes M_\fp)\le 1$ and $TR(B\otimes M_\fq)\le 1,$  there is a unital \hm\,
$\phi_\fp: A\otimes M_\fp\to B\otimes M_\fp$ and $\psi_\fq: A\otimes M_\fq\to B\otimes M_\fq$ such that
$$
\mathrm{Ell}(\phi_\fp)=\Lambda_\fp\andeqn \mathrm{Ell}(\psi_\fq)=\Lambda_\fq.
$$
%It follows from ? of \cite{Lnclasn} that there exists ({\bf I'm not sure what to put in here})

 Put $\phi=\phi_{\mathfrak{p}}\otimes {\rm
id}_{M_{\mathfrak{q}}}: A\otimes Q\to B\otimes Q$ and
$\psi=\psi_{\mathfrak{q}}\otimes {\rm
id}_{M_{\mathfrak{q}}}: A\otimes Q\to B\otimes Q.$

Note that
$$
(\phi)_{*i}=(\psi)_{*i}\,\,{\rm (} i=0,1 {\rm )} \andeqn
\phi_T=\psi_T.
$$
(they are induced by $\Gamma$). Note that $\phi_T$ and $\psi_T$
are affine homeomorphisms. Since $K_{*i}(B\otimes Q)$ is
divisible, we in fact have $[\phi]=[\psi]$ (in $KK(A\otimes Q,
B\otimes Q)$). It follows from Lemma 11.4 of \cite{Lnclasn} that there is an
automorphism $\bt: B\otimes Q\to B\otimes Q$ such that
$$
[\bt]=[{\rm id}_{B\otimes Q}]\,\,\,KK(B\otimes Q, B\otimes Q)
$$
such that $\phi$ and $\bt\circ \psi$ are asymptotically unitarily
equivalent.  Since $K_1(B\otimes Q)$ is divisible,
$H_1(K_0(A\otimes Q), K_1(B\otimes Q))=K_1(B\otimes Q).$ It
follows that $\phi$ and $\bt\circ \psi$ are strongly
asymptotically unitarily equivalent. Note also in this case
$$
\bt_T=({\rm id}_{B\otimes Q})_T.
$$
Let $\imath: B\otimes M_{\mathfrak{q}}\to B\otimes Q$ defined by
$\imath(b)=b\otimes 1$ for $b\in B.$ We consider the pair
$\bt\circ \imath\circ \phi_{\mathfrak{q}}$ and $\imath \circ
\phi_{\mathfrak{q}}.$ By applying 11.5 of \cite{Lnclasn}, there exists an
automorphism $\af: \phi_\fq(A\otimes M_{\mathfrak{q}})\to \phi_\fq(A\otimes
M_{\mathfrak{q}})$ such that $\imath\circ \af\circ
\psi_{\mathfrak{q}}$ and $\bt\circ \imath\circ
\psi_{\mathfrak{q}}$ are asymptotically unitarily equivalent (in
$M(B\otimes Q)$). So they are strongly asymptotically unitarily
equivalent. Moreover,
$$
[\af]=[{\rm id}_{B\otimes M{\mathfrak{q}}}]\,\,\,{\rm in}\,\,\,
KK(B\otimes M_{\mathfrak{q}},B\otimes M_{\mathfrak{q}}).
$$

We will show that $\bt\circ \psi$ and $\af\circ
\phi_{\mathfrak{q}}\otimes {\rm id}_{M_{\mathfrak{p}}}$ are strongly
asymptotically unitarily equivalent. Define $\bt_1=\bt\circ
\imath\circ\psi_{\mathfrak{q}}\otimes {\rm id}_{M_{\mathfrak{p}}}:
B\otimes Q\otimes M_{\mathfrak{p}}\to B\otimes Q\otimes
M_{\mathfrak{p}}.$ Let $j: Q\to Q\otimes M_{\mathfrak{p}}$ defined
by $j(b)=b\otimes 1.$ There is an isomorphism $s:
M_{\mathfrak{p}}\to M_{\mathfrak{p}}\otimes M_{\mathfrak{p}}$ with
$({\rm id}_{M_{\mathfrak{q}}}\otimes s)_{*0}=j_{*0}.$ In this case
$[{\rm id}_{M_{\mathfrak{q}}}\otimes s]=[j].$ Since
$K_1(M_{\mathfrak{p}})=0.$ By 7.2 of \cite{Lnclasn}, ${\rm
id}_{M_{\mathfrak{q}}}\otimes s$ is strongly asymptotically
unitarily equivalent to $j.$ It follows that $\af\circ
\psi_{\mathfrak{q}}\otimes {\rm id}_{M_{\mathfrak{p}}}$ and
$\bt\circ \imath\circ \psi_{\mathfrak{q}}\otimes {\rm
id}_{M_{\mathfrak{p}}}$ are strongly asymptotically unitarily
equivalent. Consider the \SCA\, $C=\bt\circ \psi(1\otimes
M_{\mathfrak{p}})\otimes M_{\mathfrak{p}}\subset B\otimes Q\otimes
M_{\mathfrak{p}}.$ In $C,$  $\bt\circ \phi|_{1\otimes
M_{\mathfrak{p}}}$ and $j_0$ are strongly asymptotically unitarily
equivalent, where $j_0: M_{\mathfrak{p}}\to C$ by $j_0(a)=1\otimes
a$ for all $a\in M_{\mathfrak{p}}.$ There exists a continuous path
of unitaries $\{v(t): t\in [0,\infty)\}\subset C$ such that
\beq\label{CM1-1} \lim_{t\to\infty}{\rm ad}\, v(t) \circ\bt\circ
\phi(1\otimes a)=1\otimes a\tforal a\in M_{\mathfrak{p}}. \eneq It
follows that $\bt\circ \psi$ and $\bt_1$ are strongly asymptotically
unitarily equivalent. Therefore $\bt\circ \psi$ and $\af\circ
\psi_{\mathfrak{q}}\otimes {\rm id}_{M_{\mathfrak{p}}}$ are strongly
asymptotically unitarily equivalent. Finally, we conclude that
$\af\circ \psi_{\mathfrak{q}}\otimes {\rm id}_{\mathfrak{p}}$ and
$\phi$ are strongly asymptotically unitarily equivalent. Note that
$\af\circ \psi_{\mathfrak{q}}$ is an isomorphism which induces
$\Gamma_{\mathfrak{q}}.$

Let $\{u(t): t\in [0,1)\}$ be a continuous path of unitaries in
$B\otimes Q$ with $u(0)=1_{B\otimes Q}$ such that
$$
\lim_{t\to\infty}{\rm ad}\, u(t)\circ \phi(a)=\af\circ
\psi_{\mathfrak{q}}\otimes {\rm id}_{M_{\mathfrak{q}}}(a)\tforal
a\in A\otimes Q.
$$
One then obtains a unitary suspended $C([0,1])$-unital \hm\,  which lifts
$\Gamma$ along $Z_{p,q}$ (see \cite{Winter-Z}).
\end{proof}

\begin{thm}[cf. Proposition 4.6 of \cite{Winter-Z}]\label{extT}
Let $A$ and $B$ be two ${\cal Z}$-stable  \CA\, in ${\cal A}.$
Suppose that there exists a strictly positive unital \hm\,
$$
\Lambda: \mathrm{Ell}(A)\to \mathrm{Ell}(B).
$$
Then there exists  a unital \hm\, $\phi: A\to B$ such that
$\phi$ induces $\Lambda.$
\end{thm}
\begin{proof}
The proof is a simple modification of that of Proposition 4.6 of
\cite{Winter-Z} by applying \ref{ext}.

First, it is clear that the proof of Lemma 4.3 of \cite{Winter-Z}
holds if isomorphism is changed to \hm\, without any changes when
both $A$ and $B$ are assumed to be simple. Moreover, the one-sided
version of Proposition 4.4 also holds. In particular the part (ii)
of that proposition holds. It follows that the \hm\, version of
Proposition 4.5 of \cite{Winter-Z} holds since proof requires no
changes except that we change the word "isomorphism" to ``\hm" twice
in the proof.

To prove this theorem, we apply \ref{ext} to obtain a unitarily
suspended $C([0,1])$-\hm\, $\phi: A\otimes B\otimes {\cal Z}$ which
has the properties described in \ref{ext}. The rest of the proof is
just a copy of the proof of Proposition 4.6 with only four  changes:
(1) ${\tilde \phi}: A\otimes {\cal Z}\to B\otimes {\cal Z}\otimes
{\cal Z}$ is a \hm\, (instead of isomorphism); (2) in the diagram
(65), $\lambda,$ $\lambda\otimes {\rm id}$ are \hm s (instead of
isomorphism); (3) ${\tilde \phi}_+$ is a \hm\, (instead of
isomorphism); (4) since $\lambda$ and ${\tilde \phi}_*$ agree as \hm
s (instead of isomorphisms) and both are order \hm s preserving the
order units, they also have to agree as such.
\end{proof}

\begin{df}\label{LASH}
A C*-algebra $A$ is said to be locally approximated by subhomogeneous C*-algebras if for any finite subset $\mathcal F\subseteq A$ and any $\ep>0$, there is a C*-subalgebra $H\subseteq A$ isomorphic to a subhomogeneous algebra such that $\mathcal F\subseteq_\ep H$.
\end{df}

\begin{rem}
It is clear from the definition that any inductive limit of locally approximately subhomogeneous C*-algebras is again a locally approximately subhomogeneous C*-algebra.
\end{rem}

\begin{lem}\label{tracecl}
Let  that $(G_0, (G_0)_+,u)$ be a countable partially ordered
weakly unperforated and rationally Riesz group, let $G_1$ be  a countable abelian group, let
$T$ be a metrizable Choquet simplex and let $\lambda_T: T\to S_u(G_0)$ be a surjective
affine continuous map sending extremal points to extremal points.
Suppose that $S_u(G_0)$ and $T$ have finitely many extremal points.

Then
there exists one unital ${\cal Z}$-stable \CA\, $A\in {\cal A}$ such that
$$
\mathrm{Ell}(A)=((G_0, (G_0)_+, [1_A]), G_1, T, \lambda_T).
$$  Moreover, the \CA\ $A$ can be locally approximated by subhomogeneous C*-algebras.
\end{lem}
\begin{proof}
Denote by $e_1, e_2, ..., e_n$ the extreme points of $S_u(G_0)$, and denote by $S_1,...,S_n$ the preimage of
$e_1,..., e_n$ under $\lambda$. Then, each $S_i$ is a face of $T$, and hence a simplex with finitely many
 extreme points. In each $S_i$, choose an extreme point $f_i$.

Set an affine map $\alpha: S_u(G_0)\dot{\times}T\to S_u(G_0)$ by
$$\alpha((e_i, g_j))=e_i,$$ where $e_i$ is an extreme point of
$S_u(G_0)$ and $g_j$ is an extreme point of $T$.
Define an affine map $\pi: S_u(G_0)\dot{\times} T\to T$ by
$\pi((e_i, g_j))=g_j$ if $g_j \in S_k$, and $\pi(e_i, g_j)=f_i$ if
$g_j$ is not in any of $S_k.$ Since there are only finitely many extreme points
in both $S_u(G_0)$ and $T,$ $\pi$ is a continuous affine surjective
map.

Then
$$\lambda_T\circ\pi=\alpha.$$
Choose a lifting $\iota: T\to S_u(G_0)\dot{\times}T$ of $\pi$ by
$\imath (g_j)=(\lambda_T(g_j), g_j)$ for $g_j\in S_j,$
$j=1,2,...,n.$ In particular, $\pi\circ\iota=\mathrm{id}_{T}.$
Define an affine map $\beta: S_u(G_0)\dot{\times}T\to
S_u(G_0)\dot{\times}T$ by $\beta=\iota\circ\pi$.

By Theorem \ref{R3}, there is a $\mathcal Z$-stable ASH-algebra
$A'\in\mathcal{A}$ with $$\mathrm{Ell}(A')=(G_0, G_1,
S_u(G_0)\dot{\times}T, \alpha).$$ By Theorem \ref{extT}, there is a
unital homomorphism $\phi: A'\to A'$ such that
$$[\phi]_0=\mathrm{id},\quad[\phi]_1=\mathrm{id},\quad\textrm{and}\quad(\phi)^{\sharp}=\beta.$$ (The compatibility between the map $\beta$ and $[\phi]_0$ follows from the commutative diagram below.)
Let  $A_n=A'$ and let $\phi_n: A_n\to A_{n+1}$ be defined by
$\phi_n=\phi.$ $n=1,2,....$  Put $A=\lim_{n\to\infty}(A_n,\phi_n).$
Since each $A_n$ is simple so is $A.$ By Theorem 11.10 of \cite{Lnclasn}, $A\in {\cal A}.$ Since each $A_n$ is an ASH-algebra, the \CA\ $A$ can be locally approximated by subhomogeneous \CA s.

Since the diagram
\begin{displaymath}
\xymatrix{
S_u(G_0)\dot{\times} T\ar[dd]_{\alpha} & & S_u(G_0)\dot{\times}T \ar[ll]_{\beta}
\ar[dl]_{\pi}\ar[dd]^{\alpha}\\
&T\ar[ul]_{\iota}\ar[dl]_{\lambda_{T}}\ar[dr]^{\lambda_T}&\\
S_u(G_0)\ar@{=}[rr] & & S_u(G_0)
}
\end{displaymath}
commutes (the left triangle commutes because
$\alpha\circ\iota=\lambda_T\circ\pi\circ\iota=\lambda_T\circ\mathrm{id}_T=\lambda_T$),
one has that the inductive limit $A=\varinjlim(A', \phi)$ satisfies
$$\mathrm{Ell}(A)=(G_0, G_1, T, \lambda_T),$$ as desired.
\end{proof}

\begin{lem}\label{decomp}
Let $G$ be a countable rationally Riesz group and let $T$ be a
metrizable Choquet simplex. Let $\lambda: T\to S_u(G)$ be a
surjective affine map preserving extreme points. Then, there are
decompositions $G=\varinjlim(G_n, \psi_n)$,
$\mathrm{Aff}T=\varinjlim(\mathbb R^{k_n}, \eta_n)$, and maps
$\lambda_n: G_n\to\mathbb R^{k_n}$ such that each $S_u(G_n)$ is a
simplex with finitely many extreme points, and the following diagram
commutes:
\begin{displaymath}
\xymatrix{
\mathbb R^{k_n}\ar[r]^{\eta_n} & \mathbb R^{k_{n+1}}\ar[r]&\cdots\ar[r] & \mathrm{Aff}T\\
G_n\ar[r]_{\psi_n}\ar[u]_{\lambda_n} & G_{n+1}\ar[u]_{\lambda_{n+1}}\ar[r]&\cdots\ar[r] & G\ar[u]_{\lambda^*}
}.
\end{displaymath}
\end{lem}
\begin{proof}
Consider the ordered group $H=G\otimes \Q$. It is clear that the ordered group $G_\mathrm{F}:=G/\mathrm{Tor}(G)$ (with positive cone the image of the positive cone of $G$) is a sub-ordered-group of $H$. Since $G$ is a rationally Riesz group, the group $H$ is a Riesz group. It follows from Effros-Handelman-Shen Theorem that there is a decomposition $H=\varinjlim(H_i, \phi_i)$, where $H_i=\mathbb Z^{m_i}$ with the usual order for some natural number $m_i$. We may assume that the images of $(\phi_{i, \infty})$ is increasing in $H$. Set
$$G'_i=\phi_{i, \infty}^{-1}(G_\mathrm{F}\cap\phi_{i, \infty}(H_i))\subseteq H_i.$$
Then the inductive limit decomposition of $H$ induces an inductive limit decomposition
$$G_\mathrm{F}=\varinjlim(G'_i, \phi_i|_{G'_i}).$$

We assert that $S_u(G'_i)=S_u(H_i)$. To show the assertion, it is enough to show that any state on $H_i$ is determined by its restriction to $G'_i$, and it is enough to show that two states of $H_i$ are same if their restrictions to $G'_i$ are same. Indeed, let $\tau_1$ and $\tau_2$ be two states of $H_i$ with same restrictions to $G'_i$. For any element $h\in H_i$, consider $\phi_{i, \infty}(h)\in H$. We then can write
$$\phi_{i, \infty}(h)=\frac{p}{q}g,$$
 for some $g\in G_{\mathrm F}$ and relatively prime numbers $p$ and $q$. In particular,
$$pg=q\phi_{i, \infty}(h)\in\phi_{i, \infty}(H_i)\cap G_\mathrm{F},$$
and hence $$qh\in G'_i.$$ Therefore, one has
$$\tau_1(h)=\frac{1}{q}\tau_1(qh)=\frac{1}{q}\tau_2(qh)=\tau_2(h).$$
This proves the assertion.

Since $S_u(H_i)$ is a finite dimensional simplex, the convex
set  $S_u(G'_i)$ is also a simplex with finitely many extreme
points. Hence we have the inductive decomposition
$G_\mathrm{F}=\varinjlim(G'_i, \psi'_i)$, where
$\psi'_i=\phi_i|_{G'_i}$.

Consider the extension
\begin{displaymath}
\xymatrix{
0\ar[r]&\mathrm{Tor}(G)\ar[r]^-\iota &G\ar[r]^-{\pi}&G_\mathrm{F}\ar[r]&0,
}
\end{displaymath}
and write $\mathrm{Tor}(G)=\varinjlim(T_i, \iota_i)$, where $T_i$
 are  finite abelian groups. Since the torsion free abelian
group $G'_1$ is finitely generated, there is a lifting $\gamma_1:
G'_1\to G$ with $\pi\circ\gamma_1(g)=\psi'_{i, \infty}(g)$ for any
$g\in G'_1$. Since an element in $G$ is positive if and only if it
is positive in the quotient $G_\mathrm{F}$, it is clear that
$\gamma_1$ is positive. Consider the ordered group $T_1\oplus G'_1$
with the order determined by $G'_1$, and set the map $\psi_{1,
\infty}: T_1\oplus G'_1 \to G$ by $(a, b)\mapsto
\iota(a)+\gamma_1(b)$. It is clear that $\psi_{1, \infty}$ is
positive.

Using the same argument, one has a positive lifting
$\gamma_2: G'_2\to G$. Since $\psi'_{i, \infty}(G'_1)\subseteq \psi'_{2, \infty}(G'_2)$, one has that $\gamma_2(g)-\gamma_1(g)\in\mathrm{Tor}(G)$ for each $g\in G'_1$. By truncating the sequence $(T_i)$, one may assume that $(\gamma_2-\gamma_1)(G'_1)\in T_2$. Define the map $\psi_{2, \infty}:  T_2\oplus G'_2 \to G$ by $(a, b)\mapsto \iota(a)+\gamma_2(b)$. Then it is clear that $\psi_{1, \infty}(T_1\oplus G'_1)\subseteq \psi_{2, \infty}(T_2\oplus G'_2)$. Define the map $\psi_{1, 2}: T_1\oplus G'_1\to T_2\oplus G'_2$ by $$(a, b)\mapsto (\iota(a)+\gamma_1(b)-\gamma_2(b), \psi'_{1, 2}(b)).$$ A direct calculation shows that $\psi_{1, \infty}=\psi_{2, \infty}\circ\psi_{1, 2}$.

Repeating this argument and setting $G_i=T_i\oplus G_i$, one has the inductive limit decomposition
$$G=\varinjlim_i(G_i, \psi_{i, i+1}).$$ Noting that the order on $G_i$ is determined by the order on
$G'_i$, one has that $S_u(G_i)=S_u(G_i')$, and hence the convex
set $S_u(G_i)$ is a simplex with finitely many extreme
points.

Let $\{a_n\}$ be a dense sequence in the positive cone $\mathrm{Aff}^+T$. Consider the map $\lambda^*\circ \psi_{i, \infty}: G\to \mathrm{Aff} T$.
Since $S_u(G_i)=S_u(H_i)$ and $H_i=\mathbb Z^{m_i}$, the image of
positive elements of $G_i$ is contained in a finite dimensional
cone. Since images of $G_i$ are increasing, we may choose $\{b_1,
..., b_i,...\}\subseteq \mathrm{Aff}T$ and natural numbers
$n_1<\cdots<n_i<\cdots$ such that $\{b_1,...,b_{n_i}\}$ is a set of
generators for the image of $G_i$ in $\mathrm{Aff}T$.

For each $i$, set $k_i=i+n_i$. We identify the affine space
$$\mathbb R^{k_i}\cong(\mathbb Ra_1\oplus\cdots\oplus\mathbb Ra_i)\oplus (\mathbb Rb_1\oplus
\cdots\oplus\mathbb Rb_{n_i})$$
as the subspace of $\mathrm{Aff}T$ spanned by $a_1,...,a_i, b_1, ..., b_{n_i}$. Define the map $\lambda_i: G_i\to\mathbb R^{k_i}$ by $$g\mapsto(0\oplus\cdots\oplus0)\oplus (\lambda^*\circ \psi_{i, \infty}(g)),$$
the map $\eta_i:\mathbb R^{k_i}\to \mathbb R^{k_{i+1}}$ by
$$(f, g)\mapsto \iota_1(f)\oplus\iota_2(g),$$
where $\iota_1$ and $\iota_2$ are the inclusions of $\mathbb Ra_1\oplus\cdots\oplus\mathbb Ra_i$ and
$\mathbb Rb_1\oplus\cdots\oplus\mathbb Rb_{n_i}$ to $\mathbb Ra_1\oplus\cdots\oplus\mathbb Ra_{i+1}$ and
$\mathbb Rb_1\oplus\cdots\oplus\mathbb Rb_{n_{i+1}}$ in $\mathrm{Aff}T$, respectively. Then, it is a
straightforward calculation that $\mathrm{Aff}T$ has the decomposition $\varinjlim(\mathbb R^{k_i}, \eta_i)$,
and the diagram in the lemma commutes.
\end{proof}

Finally, we reach the main result of this paper.

\begin{thm}\label{MT}
Let $(G_0, (G_0)_+,u)$ be a countable partially ordered
weakly unperforated and rationally Riesz group, let $G_1$ be  a countable abelian group, let
$T$ be a metrizable Choquet simplex and let $\lambda_T: T\to S_u(G_0)$ be a surjective
affine continuous map sending extremal points to extremal points. Then
there exists one (and exactly one, up to isomorphic) unital ${\cal Z}$-stable \CA\, $A\in {\cal A}$
such that
$$
\mathrm{Ell}(A)=((G_0, (G_0)_+, u), G_1, T, \lambda_T).
$$
Moreover, $A$ can be constructed to be locally approximated by subhomogeneous \CA s.
\end{thm}

\begin{proof}
Note that the part of the statement about ``exactly one, up to
isomorphism" follows from \cite{Lnclasn}.

By Lemma \ref{decomp}, there exists a decomposition
\begin{displaymath}
\xymatrix{
\mathbb R^{k_n}\ar[r]^{\eta_n} & \mathbb R^{k_{n+1}}\ar[r]&\cdots\ar[r] & \mathrm{Aff}T\\
K_n\ar[r]_{\psi_n}\ar[u]_{\lambda_n} & K_{n+1}\ar[u]_{\lambda_{n+1}}\ar[r]&\cdots\ar[r] & G_0\ar[u]_{\lambda_T^*}
},
\end{displaymath}
where each $K_n$ is a rationally Riesz group with $S_u(K_n)$ having
finitely many extreme points. By Lemma \ref{tracecl}, there is a
unital $\mathcal Z$-stable algebra $A_n\in\mathcal A$ such that
$$
((K_0(A_n), K_0(A_n)_+, [1_{A_n}]), K_1(A_n),
\mathrm{Aff}(\mathrm{T}(A_n)), \lambda_{A_n})= ((K_n, (K_n)_+ , u), G_1,\mathbb
R^{k_n}, \lambda_n),$$ and each $A_n$ can be locally approximated by subhomogeneous C*-algebras.

By Theorem \ref{extT}, there are  *-homomorphisms $\phi_n: A_n\to
A_{n+1}$ such that $(\phi_n)_{*0}=\psi_n,$ $(\phi_n)_{*1}={\rm
id}_{G_1}$ and $(\psi_n)_*=\eta_n$, where $(\psi_n)_*$ is the
induced map from
$\mathrm{Aff}(\mathrm{T}(A_n))\to\mathrm{Aff}(\mathrm{T}(A_{n+1}))$.
Then the inductive limit $A=\varinjlim_n(A_n, \psi_n)$ is in the
class $\mathcal A$ and satisfies $$\mathrm{Ell}(A)=((G_0, (G_0)_+, u),
G_1, T, \lambda_T).$$ Since each $A_n$ can be locally approximated by subhomogeneous C*-algebras, so does the \CA\ $A$.
\end{proof}

\bibliographystyle{amsplain}

%\bibliography{operator_algebras}

\begin{thebibliography}{10}

\bibitem{Bla-Ktheory}
B.~Blackadar, \emph{$\mbox{K}$-theory for operator algebras}, second ed.,
  Mathematical Sciences Research Institute Publications, no.~5, Cambridge
  University Press, 1998.

\bibitem{Dix}
J.~Dixmier, \emph{On some $\textrm{C*}$-algebras considered by
  $\textrm{Glimm}$}, J. Funct. Anal. \textbf{1} (1967), 182--203.

\bibitem{EHS}
E.~G. Effros, D.~E. Handelman, and C.~L. Shen, \emph{Dimension groups and their
  affine representations}, Amer. J. Math \textbf{102} (1980), no.~2, 385--407.

\bibitem{Ell-AI}
G.~A. Elliott, \emph{A classification of certain simple $\mbox{C}$*-algebras.},
  Quantum and non-commutative analysis (Kyoto, 1992), Math. Phys. Stud.,
  no.~16, Kluwer Acad. Publ., Dordrecht, 1993, pp.~373--385.

\bibitem{point-line}
\bysame, \emph{An invariant for simple $\mbox{C}$*-algebras}, Canadian
  Mathematical Society. 1945--1995 \textbf{3} (1996), 61--90.

\bibitem{EGL-ADiv}
G.~A. Elliott, G.~Gong, and L.~Li, \emph{Approximate divisibility of simple
  inductive limit $\textrm{C*}$-algebras}, Operator algebras and operator
  theory (Shanghai, 1997), Contemp. Math., vol. 228, Amer. Math. Soc.,
  Providence, RI, 1998, pp.~87--97.

\bibitem{EGL-AH}
\bysame, \emph{On the classification of simple inductive limit
  $\mbox{C}^*$-algebras, $\mbox{II}$: The isomorphism theorem}, Invent. Math.
  \textbf{168} (2007), no.~2, 249--320.

\bibitem{EN-Tapprox}
G.~A. Elliott and Z.~Niu, \emph{On tracial approximation}, J. Funct. Anal.
  \textbf{254} (2008), no.~2, 396--440.

\bibitem{GJS-Z}
G.~Gong, X.~Jiang, and H.~Su, \emph{Obstructions to $\textrm{$\mathcal
  Z$}$-stability for unital simple $\textrm{C*}$-algebras}, Canad. Math. Bull.
  \textbf{43} (2000), no.~4, 418--426.

\bibitem{Goodearl}
K.~R. Goodearl, \emph{Partially ordered abelian groups with interpolation},
  Mathematical Surveys and Monographs, no.~20, American Mathematical Society,
  Providence, RI, 1986.

\bibitem{JS-Z}
X.~Jiang and H.~Su, \emph{On a simple unital projectionless
  $\textrm{C*}$-algebra}, Amer. J. Math \textbf{121} (1999), no.~2, 359--413.

\bibitem{KP0}
E.~Kirchberg and N.~C. Phillips, \emph{Embedding of exact
  $\textrm{C*}$-algebras in the $\textrm{Cuntz}$ algebra $\textrm{$\mathcal
  O_2$}$}, J. Reine Angew. Math. \textbf{525} (2000), 17--53.

\bibitem{Lin-book}
H.~Lin, \emph{An introduction to the classification of amenable
  $\textrm{C*}$-algebras}, World Scientific Publishing Co., River Edge, NJ,
  2001.

\bibitem{Lnduke}
\bysame, \emph{Classification of simple $\mbox{C}$*-algebras of tracial
  topological rank zero}, Duke Math. J. \textbf{125} (2004), no.~1, 91--119.

\bibitem{Lin-Corelle}
\bysame, \emph{Traces and simple $\textrm{C*}$-algebras with tracial
  topological rank zero}, J. Reine Angew. Math. \textbf{568} (2004), 99--137.

\bibitem{Lin-App}
\bysame, \emph{Localizing the $\textrm{Elliott}$ conjecture at strongly
  self-absorbing $\textrm{C*}$-algebras, $\textrm{II}$---an appendix}, arXiv:
  0709.1654v1 (2007).

\bibitem{LinTAI}
\bysame, \emph{Simple nuclear $\mbox{C}$*-algebras of tracial topological rank
  one}, J. Funct. Anal. \textbf{251} (2007), no.~2, 601--679.

\bibitem{Lnclasn}
\bysame, \emph{Asymptotically unitary equivalence and classification of simple
  amenable $\textrm{C*}$-algebras}, arXiv: 0806.0636 (2008).

\bibitem{Myg-CJM}
J.~Mygind, \emph{Classification of certain simple $\textrm{C*}$-algebras with
  torsion $\mathrm{K}_1$}, Canad. J. Math. \textbf{53} (2001), no.~6,
  1223--1308.

\bibitem{Niu-thesis}
Z.~Niu, \emph{A classification of certain tracially approximately
  subhomogeneous $\mbox{C}$*-algebras}, Ph.D. thesis, University of Toronto,
  2005.

\bibitem{Ph1}
N.~C. Phillips, \emph{A classification theorem for nuclear purely infinite
  simple $\textrm{C*}$-algebras}, Doc. Math. \textbf{5} (2000), 49--114.

\bibitem{RorUHF}
M.~R{\o}rdam, \emph{On the structure of simple $\mbox{C}$*-algebras tensored
  with a $\mbox{UHF}$ algebra}, J. Funct. Anal. \textbf{100} (1991), 1--17.

\bibitem{RorUHF2}
\bysame, \emph{On the structure of simple $\mbox{C}$*-algebras tensored with a
  $\mbox{UHF}$ algebra. $\mbox{II}$}, J. Funct. Anal \textbf{107} (1992),
  no.~2, 255--269.

\bibitem{Ror-infproj}
\bysame, \emph{A simple $\textrm{C*}$-algebra with a finite and an infinite
  projection}, Acta Math. \textbf{191} (2003), no.~1, 109--142.

\bibitem{Sakai-book}
S.~Sakai, \emph{$\textrm{C*}$-algebras and $\textrm{W*}$-algebras}, Ergebnisse
  der Mathematik und ihrer Grenzgebiete, Band 60, Springer-Verlag, New
  York-Heidelberg, 1971.

\bibitem{Toms-Ann}
A.~Toms, \emph{On the classification problem for nuclear
  $\textrm{C*}$-algebras}, Ann. of Math. (2) \textbf{167} (2008), 1059--1074.

\bibitem{TW1}
A.~Toms and W.~Winter, \emph{$\textrm{$\mathcal{Z}$}$-stable $\textrm{ASH}$
  algebras}, arXiv:math/0508218 (2005).

\bibitem{Vill-EHS0}
J.~Villadsen, \emph{The range of the $\textrm{Elliott}$ invariant}, J. Reine
  Angew. Math. \textbf{462} (1995), 31--55.

\bibitem{Vill-EHS}
\bysame, \emph{The range of the $\textrm{Elliott}$ invariant of the simple
  $\textrm{AH}$-algebras with slow dimension growth}, K-Theory \textbf{15}
  (1998), no.~1, 1--12.

\bibitem{Winter-Z}
W.~Winter, \emph{Localizing the $\textrm{Elliott}$ conjecture at strongly
  self-absorbing $\textrm{C*}$-algebras}, arXiv: 0708.0283v3 (2007).

\end{thebibliography}

\providecommand{\bysame}{\leavevmode\hbox to3em{\hrulefill}\thinspace}
\providecommand{\MR}{\relax\ifhmode\unskip\space\fi MR }
\providecommand{\MRhref}[2]{%
  \href{http://www.ams.org/mathscinet-getitem?mr=#1}{#2}
}
\providecommand{\href}[2]{#2}

\end{document}